\documentclass[a4paper,12pt]{article}
\usepackage[T1]{fontenc}
\usepackage[english]{babel}
\usepackage[ansinew]{inputenc}
\usepackage{amsmath,amssymb,amsfonts,amstext}
\usepackage{graphicx}
\usepackage{amsthm}

\newtheorem{theorem}{Theorem}[section]

\newtheorem{lemma}{Lemma}[section]

\numberwithin{equation}{section}
\begin{document}

\title{Extremal Polynomials and Entire Functions of Exponential Type}
\author{Michael Revers}
\maketitle

\begin{abstract}
In this paper, we discuss asymptotic relations for the approximation of
$\left\vert x\right\vert ^{\alpha},\alpha>0$ in $L_{\infty}\left[
-1,1\right] $ by Lagrange interpolation polynomials based on the zeros of the
Chebyshev polynomials of first kind. The limiting process reveals an entire
function of exponential type for which we can present an explicit formula. As
a consequence, we further deduce an asymptotic relation for the approximation
error when $\alpha\rightarrow\infty$. Finally, we present connections of our
results together with some recent work of Ganzburg \cite{Ganzburg1} and Lubinsky \cite{Lubinsky1}, by presenting
numerical results, indicating a possible constructive way towards a
representation for the Bernstein constants.

\textbf{MSC Classification (2010):} 41A05, 41A10, 41A60, 65D05


\end{abstract}

\section{The Bernstein Constant}

Let $\alpha>0$ be not an even integer. Starting in year 1913 for the case
$\alpha=1$, and later in 1938 for the general case $\alpha>0$, S.N. Bernstein
\cite{Bernstein1}, \cite{Bernstein3} established the limit%
\[
\Delta_{\infty,\alpha}=\lim_{n\rightarrow\infty}n^{\alpha}E_{n}\left(
\left\vert x\right\vert ^{\alpha},L_{\infty}\left[ -1,1\right] \right) ,
\]
where%
\[
E_{n}\left( f,L_{p}\left[ a,b\right] \right) =\inf\left\{ \left\Vert
f-p\right\Vert _{L_{p}\left[ a,b\right] }:\deg\left( p\right) \leq
n\right\}
\]
denotes the error in best $L_{p}$ approximation of a function $f$ on the
interval $\left[ a,b\right] $ by polynomials of degree less or equal $n$.
The proofs in \cite{Bernstein1}, \cite{Bernstein3} are highly difficult and
long, missing many non-trivial technical details. In his 1938 paper, Bernstein
made essential use of the homogeneity property of $\left\vert x\right\vert
^{\alpha}$, namely that for $c>0$ one has $\left\vert cx\right\vert ^{\alpha
}=c^{\alpha}\left\vert x\right\vert ^{\alpha}$. Using this property, one gets
for $a,b>0$ and all $1\leq p\leq\infty$ the relation (see \cite{Lubinsky1},
Lemma 8.2)%
\begin{equation}
E_{n}\left( \left\vert x\right\vert ^{\alpha},L_{p}\left[ -b,b\right]
\right) =\left( \frac{b}{a}\right) ^{\alpha+\frac{1}{p}}E_{n}\left(
\left\vert x\right\vert ^{\alpha},L_{p}\left[ -a,a\right] \right) .
\label{RelateEn}%
\end{equation}
This enabled Bernstein to relate the uniform best approximating error on
$\left[ -1,1\right] $ to that on $\left[ -n,n\right] $. A routine argument
shows that identity (\ref{RelateEn}) sends the best approximating polynomials
$P_{n}^{\ast}$ of order $n$ with respect to $\left[ -1,1\right] $ into a sequence
$\left\{ n^{\alpha}P_{n}^{\ast}\left( \frac{\cdot}{n}\right): n=1, 2, \ldots \right\}$ 
of scaled polynomials in $\left[ -n,n\right] $. Bernstein
also established a formulation of the limit as the error in approximation on
the real line by entire functions of exponential type, namely,%
\begin{align*}
\Delta_{\infty,\alpha} & =\lim_{n\rightarrow\infty}n^{\alpha}E_{n}\left(
\left\vert x\right\vert ^{\alpha},L_{\infty}\left[ -1,1\right] \right) \\
& =\lim_{n\rightarrow\infty}E_{n}\left( \left\vert x\right\vert ^{\alpha
},L_{\infty}\left[ -n,n\right] \right) \\
& =\lim_{n\rightarrow\infty}\left\Vert \left\vert x\right\vert ^{\alpha
}-n^{\alpha}P_{n}^{\ast}\left( \frac{\cdot}{n}\right) \right\Vert _{L_{\infty
}\left[ -n,n\right] }\\
& =\inf\left\{ \left\Vert \left\vert x\right\vert ^{\alpha}-H\right\Vert
_{L_{\infty}\left( \mathbb{R}\right) }:H\text{ is entire of exponential
type}\leq1\right\} .
\end{align*}
Recall that an entire function $f$ is of exponential type $A\geq0$ means that
for each $\varepsilon>0$ there is $z_{0}=z_{0}\left( \varepsilon\right) $, such that
\begin{equation}
\left\vert f\left( z\right) \right\vert \leq\exp\left( \left\vert
z\right\vert \left( A+\varepsilon\right) \right) ,\quad\forall
z\in\mathbb{C}:\left\vert z\right\vert \geq\left\vert z_{0}\right\vert .
\label{Entire1}%
\end{equation}
Moreover, $A$ is taken to be the infimum over all possible numbers for which
(\ref{Entire1}) holds. The elegant formulation which introduces now functions
of exponential type extends to spaces other than $L_{\infty}$. Ganzburg
\cite{Ganzburg1} and Lubinsky \cite{Lubinsky1} have shown that for all $1\leq
p\leq\infty$ positive constants $\Delta_{p,\alpha}$ exists, where
$\Delta_{p,\alpha}$ is defined by%
\begin{align}
\Delta_{p,\alpha} & =\lim_{n\rightarrow\infty}n^{\alpha+\frac{1}{p}}%
E_{n}\left( \left\vert x\right\vert ^{\alpha},L_{p}\left[ -1,1\right]
\right) \nonumber\\
& =\inf\left\{ \left\Vert \left\vert x\right\vert ^{\alpha}-H\right\Vert
_{L_{p}\left( \mathbb{R}\right) }:H\text{ is entire of exponential type
}\leq1\right\} . \label{Entire2}%
\end{align}
From now on $\Delta_{p,\alpha}$ are called the Bernstein constants.

\medskip\noindent
Only for $p=1,2$ are the values $\Delta_{p,\alpha}$ known.
In 1947, Nikolskii \cite{Nikolskii} proved that%
\[
\Delta_{1,\alpha}=\frac{\left\vert \sin\frac{\alpha\pi}{2}\right\vert }{\pi
}8\Gamma\left( \alpha+1\right) \sum_{n=0}^{\infty}\frac{\left( -1\right)
^{n}}{\left( 1+2n\right) ^{\alpha+2}},\quad\alpha>-1,
\]
and in 1969, Raitsin \cite{Raitsin} established%
\[
\Delta_{2,\alpha}=\frac{\left\vert \sin\frac{\alpha\pi}{2}\right\vert }{\pi
}2\Gamma\left( \alpha+1\right) \sqrt{\frac{\pi}{2\alpha+1}},\quad
\alpha>-\frac{1}{2}.
\]
In contrast to the case of the $L_{\infty}$ norm, no single value of
$\Delta_{\infty,\alpha}$ is known. Bernstein speculated that%
\[
\Delta_{\infty,1}=\lim_{n\rightarrow\infty}nE_{n}\left( \left\vert
x\right\vert ,L_{\infty}\left[ -1,1\right] \right) =\frac{1}{2\sqrt{\pi}%
}=0.28209~47917\ldots
\]
Over the years the speculation became known as the Bernstein conjecture in
approximation theory. Some 70 years later Varga and Carpenter \cite{Varga1}, 
using sophisticated high precision scientific computational methods, calculated
the quantity numerically to%
\[
\Delta_{\infty,1}=0.28016~94990~23869\ldots
\]
Further extensive numerical explorations for the computation of $\Delta
_{\infty,\alpha}$ have been provided later by Varga and Carpenter
\cite{Varga2}. Their numerical work gave an enormous impact into the
analytical investigation of approximation problems, not only restricted to the
Bernstein constants. We would also like to mention the numerical work of
Pachón and Trefethen (\cite{PachonTrefethen}, Figure 4.4) from 2008, when they
recomputed $\left\{ nE_{n}\left( \left\vert x\right\vert ,L_{\infty}\left[
-1,1\right] \right) :n=1,\ldots,10^{4}\right\} $ again and provided an
graphical illustration indicating a monotonic growth behavior. As the story
continued, the approximation of entire functions of exponential type became a
much studied topic in function approximation, see \cite{Dryanov},
\cite{Trigub}, but also in connection to problems in number theory, see for
instance \cite{Vaaler}. As an further application in number theory, we would
like to mention a recent paper of Ganzburg \cite{Ganzburg3}, where he
discusses new asymptotic relations between Zeta-, Dirichlet- and Catalan
functions in connection with the asymptotics of Lagrange-Hermite interpolation
for $\left\vert x\right\vert ^{\alpha}$.

\medskip\noindent
Turning back to the Bernstein constants $\Delta_{p,\alpha}$,
intensive emphasis has been placed on the structure of those entire functions
of exponential type which minimize (\ref{Entire2}). For $p=1$, the (unique)
minimizing entire function of exponential type $1$ may be expressed as an
interpolation series at the nodes $\left\{ \left( j-\frac{1}{2}\right)
\pi:j=1,2,\ldots\right\} $, see (\cite{Ganzburg1}, p. 197) or
(\cite{Lubinsky2}, Formula 1.8). For $p=\infty$ an analogous interpolation
series at unknown interpolation nodes was derived by Lubinsky in
(\cite{Lubinsky2}, Theorem 1.1). In (\cite{Lubinsky1}, Theorem 1.1) he proved
the following result.

\medskip\noindent
Denote by $P_{n}^{\ast}$ the best approximating polynomial of order
$n$ to $\left\vert x\right\vert ^{\alpha}$ in the $L_{p}$ norm. Then, for all 
$1\leq p\leq\infty$, $\alpha>-\frac{1}{p}$ not an even integer, one has%
\begin{align}
\Delta_{p,\alpha} & =\lim_{n\rightarrow\infty}n^{\alpha+\frac{1}{p}%
}\left\Vert \left\vert x\right\vert ^{\alpha}-P_{n}^{\ast}\right\Vert _{L_{p}\left[
-1,1\right] }\nonumber\\
& =\lim_{n\rightarrow\infty}n^{\alpha+\frac{1}{p}}E_{n}\left( \left\vert
x\right\vert ^{\alpha},L_{p}\left[ -1,1\right] \right) \nonumber\\
& =\lim_{n\rightarrow\infty}E_{n}\left( \left\vert x\right\vert ^{\alpha
},L_{p}\left[ -n,n\right] \right) \label{Entire3}\\
& =\lim_{n\rightarrow\infty}\left\Vert \left\vert x\right\vert ^{\alpha
}-n^{\alpha}P_{n}^{\ast}\left( \frac{\cdot}{n}\right) \right\Vert _{L_{p}\left[
-n,n\right] }\nonumber\\
& =\left\Vert \left\vert x\right\vert ^{\alpha}-H_{\alpha}^{\ast}\right\Vert
_{L_{p}\left( \mathbb{R}\right) }\nonumber\\
& =\inf\left\{ \left\Vert \left\vert x\right\vert ^{\alpha}-H\right\Vert
_{L_{p}\left( \mathbb{R}\right) }:H\text{ is entire of exponential type}%
\leq1\right\} .\nonumber
\end{align}
Moreover, uniformly on compact subsets of $\mathbb{C}$,%
\[
\lim_{n\rightarrow\infty}n^{\alpha}P_{n}^{\ast}\left( \frac{z}{n}\right)
=H_{\alpha}^{\ast}\left( z\right) ,
\]
and there is exactly one entire function $H$ of exponential type $\leq1$ which
minimizes (\ref{Entire3}). While various versions of this equality and relations (\ref{Entire3}) have been discussed by Bernstein, Raitsin and Ganzburg, the uniqueness of $H_{\alpha}^{\ast}$ proved in \cite{Lubinsky1} is a highly nontrivial result.

\medskip\noindent
From the Chebyshev alternation theorem it follows that for
each integer $n$ the best approximating polynomial $P_{n}^{\ast}$ of order $n$ to
$\left\vert x\right\vert ^{\alpha}$ in the in $L_{\infty}$ norm can be
represented as an interpolating polynomial with unknown consecutive nodes in
$\left[ -1,1\right] $. Thus, if one can find something about the nature of
those best approximating interpolation nodes in $\left[ -1,1\right] $, then
we would successfully find an approach for a constructive analytical
approximation towards some representations for the Bernstein constants
$\Delta_{\infty,\alpha}$. Since $\left\vert x\right\vert ^{\alpha}$ is an even
function a standard argument allows us to restrict ourselves to interpolation
polynomials of even order $n=2m$. It is not surprising that Bernstein
\cite{Bernstein2} himself, in 1937, studied the interpolation process to
$\left\vert x\right\vert ^{\alpha}$ by using the modified Chebyshev system%
\begin{align*}
x_{0}^{\left( 2n\right) } & =0,\\
x_{j}^{\left( 2n\right) } & =\cos\frac{\left( j-1/2\right) \pi}%
{2n},\quad j=1,2,\ldots2n,
\end{align*}
where the $x_{j}^{\left( 2n\right) }$ are the zeros of the Chebyshev
polynomial $T_{2n}$ of first kind, defined by $T_{n}\left( x\right)
=\cos\left( n\arccos x\right) $. However, $x_{0}^{\left( 2n\right) }$ is
an additional choice, but not a zero of $T_{2n}$, in order to obtain the
corresponding interpolation polynomial $P_{2n}^{\left( 1\right) }$ of order
$2n$ for $\left\vert x\right\vert ^{\alpha}$. The final answer for its limit
relation was given not before 2002 by Ganzburg (\cite{Ganzburg1}, Formula
2.7). For $\alpha>0$ one has%
\begin{equation}
\lim_{n\rightarrow\infty}\left( 2n\right) ^{\alpha}\left\Vert \left\vert
x\right\vert ^{\alpha}-P_{2n}^{\left( 1\right) }\right\Vert _{L_{\infty}
\left[ -1,1\right] }=\frac{2}{\pi}\left\vert \sin\frac{\pi\alpha}%
{2}\right\vert \int_{0}^{\infty}\frac{t^{\alpha-1}}{\cosh\left( t\right)
}dt. \label{Formel1}%
\end{equation}
Let us give some remarks on equation (\ref{Formel1}). Firstly, we mention that
in \cite{Bernstein2} Bernstein himself established a slightly weaker solution
compared to formula (\ref{Formel1}). Secondly, an extension of limit relation
(\ref{Formel1}) to complex values for $\alpha$ was obtained recently in
\cite{Ganzburg2}.

\medskip\noindent It is remarkable that, since the beginning with Bernstein,
no one has studied in detail the interpolation process by using the node
system consisting of the $2n+1$ zeros of $T_{2n+1}$, since this node system
automatically includes $x=0$ as a node and apparently it seems to be the more
natural choice. To go into detail, let
\[
x_{j}^{\left( 2n+1\right) }=\cos\frac{\left( j-1/2\right) \pi}{2n+1},\quad
j=1,2,\ldots2n+1,
\]
to be the zeros of $T_{2n+1}$ and let us denote by $P_{2n}^{\left( 2\right)
} $ the corresponding interpolation polynomial of order $2n$ for $\left\vert
x\right\vert ^{\alpha}$. There is one paper \cite{Zhu}, dealing with this
node system and presenting the result that the approximation order $\left\Vert
\left\vert x\right\vert ^{\alpha}-P_{2n}^{\left( 2\right) }\right\Vert
_{L_{\infty}\left[ -1,1\right] }=O\left( 1\right) /n^{\alpha}$ when
$\alpha\in\left( 0,1\right) $. In other words, the interpolation process
attains the Jackson order. We also would like to mention a recent monograph by
Ganzburg (\cite{Ganzburg4}, Theorem 4.2.3, Corollary 4.3.2 and Remark 4.3.3)
for a more general approach to pointwise asymptotic relations within this topic.

\medskip\noindent In 2013, the author \cite{Revers1} established a strong
asymptotic formula, valid for all $\alpha>0$, from which he established an
upper estimate for the error term, see (\cite{Revers1}, Corollary 2), by
showing that%
\begin{equation}
\overline{\lim_{n\rightarrow\infty}}\left( 2n\right) ^{\alpha}\left\Vert
\left\vert x\right\vert ^{\alpha}-P_{2n}^{\left( 2\right) }\right\Vert
_{L_{\infty}\left[ -1,1\right] }\leq\frac{2}{\pi}\left\vert \sin\frac
{\pi\alpha}{2}\right\vert \int_{0}^{\infty}\frac{t^{\alpha}}{\sinh\left(
t\right) }dt, \label{Formel2}%
\end{equation}
introducing an integral of similar nature to that in formula (\ref{Formel1}).
In this paper we continue the investigation into the precise limiting quantity
of $\left( 2n\right) ^{\alpha}\left\Vert \left\vert x\right\vert ^{\alpha
}-P_{2n}^{\left( 2\right) }\right\Vert _{L_{\infty}\left[ -1,1\right] }$
for all $\alpha>0$.

\medskip\noindent
The paper is organized as follows.

\medskip\noindent
In section 2 we collect some definitions for several
constants and functions together with some standard results for later use.

\medskip\noindent
In section 3 we establish the precise limit relation
(Theorem \ref{Theorem31}) and we show that the scaled polynomials $n^{\alpha
}P_{n}^{\left( 2\right) }\left( \frac{\cdot}{n}\right) $ uniformly
converge on compact subsets of the real line to an entire function $H_{\alpha
}$ of exponential type $1$ (Theorems \ref{Theorem32} and \ref{Theorem33}). We
may also present an explicit expansion for $H_{\alpha}$ as an interpolating
series for $\left\vert x\right\vert ^{\alpha}$ (Theorem \ref{Theorem33}). As
it can be seen later from the representation for the explicit limiting error
term, i.e. from%
\begin{equation}
\lim_{n\rightarrow\infty}\left( 2n\right) ^{\alpha}\left\Vert \left\vert
x\right\vert ^{\alpha}-P_{2n}^{\left( 2\right) }\right\Vert _{L_{\infty}
\left[ -1,1\right] }=\left\Vert H\left( \alpha,\cdot\right) \right\Vert
_{L_{\infty}\left[ 0,\infty\right) }, \label{Limitresult}%
\end{equation}
the exact determination of the quantity on the right-hand side in
(\ref{Limitresult}) for individual values for $\alpha$ appears to be a rather
difficult challenge.

\medskip\noindent
In section 4 we study a certain envelope function
$H_{1}\left( \alpha,\cdot\right) $ with respect to $\left\vert H\left(
\alpha,\cdot\right) \right\vert $. We then present in Theorem \ref{Theorem41}
an asymptotic formula for $\left\Vert H_{1}\left( \alpha,\cdot\right)
\right\Vert _{L_{\infty}\left[ 0, \infty\right) }$, when $\alpha
\rightarrow\infty$, involving again the integral in formula (\ref{Formel2}).

\medskip\noindent
In section 5, by using an higher order asymptotics and
investigating into an (itself) interesting integral inequality, see Theorem
\ref{Theorem51}, we finally arrive in Theorem \ref{Theorem53} at an
asymptotic connection between $\left\Vert H_{1}\left( \alpha,\cdot\right)
\right\Vert _{L_{\infty}\left[ 0, \infty\right) }$ and $\left\Vert H\left(
\alpha,\cdot\right) \right\Vert _{L_{\infty}\left[ 0, \infty\right) }$,
when $\alpha\rightarrow\infty$.

\medskip\noindent
In Section 6, to emphasize the importance of
the interpolation formulas based on the $P_{n}^{\left( 1\right) }$ and
$P_{n}^{\left( 2\right) }$ polynomials, we present a compilation of
numerical results involving some non-trivial linear combinations of the just
mentioned polynomials together with their corresponding Chebyshev polynomials
$T_{n}$, in order to present explicit formulas for near best approximation polynomials in the $L_{\infty}$ norm, 
see formula (\ref{Nearbest}), together with their corresponding entire functions of exponential type, see formula (\ref{BestEntire}). 
Possibly and hopefully these formulas could indicate a feasible direction towards some explicit asymptotic representations of best
approximation polynomials for $\left\vert x\right\vert ^{\alpha}$ in the
$L_{\infty}$ norm and thus for the Bernstein constants $\Delta_{\infty,\alpha
}$ themselves.

\section{Notation}

In this section we record the following constants and functions, together with
properties which are used later in the paper. We denote by $\Gamma\left(
\cdot\right) $ the usual Gamma function. The Chebyshev polynomials of first kind are denoted by $T_{n}$, where $T_{n}\left( x\right) =\cos\left( n\arccos x\right) $. 
For $x\in\mathbb{R}$, let $\left[
x\right] $ to be the floor function, namely $\left[ x\right] =\max\left\{
m\in\mathbb{Z}:m\leq x\right\} $. Obviously, then $x-1<\left[ x\right] \leq
x$. We define the following constants.
\[
\begin{array}
[c]{ll}
C\left( \alpha\right) ={
{\displaystyle\int_{0}^{\infty}}
}\dfrac{t^{\alpha}}{\sinh\left( t\right) }dt, & \alpha>0,\\
Z\left( \alpha\right) =
{\displaystyle\sum_{n=1}^{\infty}}
\dfrac{1}{n^{\alpha}}, & \alpha>1.
\end{array}
\]
Next, we define the following functions.
\[
\begin{array}
[c]{ll}
H\left( \alpha,x\right) =
{\displaystyle\int_{0}^{\infty}}
\dfrac{t^{\alpha}}{\sinh\left( t\right) }\dfrac{x\sin\left( x\right)
}{x^{2}+t^{2}}dt, & \alpha>0, x>0,\\
H_{1}\left( \alpha,x\right) =
{\displaystyle\int_{0}^{\infty}}
\dfrac{t^{\alpha}}{\sinh\left( t\right) }\dfrac{x}{x^{2}+t^{2}}dt, &
\alpha>0, x>0,\\
H_{2}\left( \alpha,x\right) =
{\displaystyle\int_{0}^{\infty}}
\dfrac{t^{\alpha}}{\sinh\left( t\right) }\dfrac{x^{2}}{x^{2}+t^{2}}dt, &
\alpha>0, x>0.
\end{array}
\]
Note that $H\left( \alpha,\cdot\right) $ should not be mixed up with the
subsequent following definition of $H_{\alpha}$. We proceed further with:
\[
\begin{array}
[c]{ll}%
F\left( \alpha,x\right) =
{\displaystyle\int_{0}^{\infty}}
\dfrac{t^{\alpha}}{\sinh\left( xt\right) }\dfrac{1}{1+t^{2}}dt, &
\alpha>0,x>0,\\
G\left( \alpha,x\right) =
{\displaystyle\int_{0}^{\infty}}
t^{\alpha}e^{-xt}\dfrac{1}{1+t^{2}}dt, & \alpha>0, x>0,\\
R\left( \alpha,x\right) =\left( x/ \alpha \right) F\left( \alpha+1,x\right)
-F\left( \alpha,x\right) , & \alpha>0, x>0,\\
S\left( \alpha,x\right) =\left( \alpha x^{\alpha-1}/ 2 \right) \left( x^{2}
+\alpha^{2}\right) R\left( \alpha,x\right) , & \alpha>0, x>0,\\
F_{1}\left( \alpha,x\right) =\left( 2-\dfrac{1}{2^{\alpha}}\right) Z\left(
\alpha+1\right) G\left( \alpha,x\right) , & \alpha>0, x>0,\\
F_{2}\left( \alpha,x\right) =\left( 2-\dfrac{1}{2^{\alpha-2}}\right)
Z\left( \alpha-1\right) G\left( \alpha,x\right) , & \alpha>2, x>0.
\end{array}
\]
We collect the following easy to establish properties.
\begin{equation}
\begin{array}
[c]{lll}
\text{(a)} & H_{1}\left( \alpha,x\right) =x^{\alpha}F\left( \alpha
,x\right) , & \alpha>0, x>0,\\
\text{(b)} & H_{2}\left( \alpha,x\right) =x^{\alpha+1}F\left(
\alpha,x\right) , & \alpha>0, x>0,\\
\text{(c)} & 0\leq H_{2}\left( \alpha,x\right) \leq C\left( \alpha\right)
, & \alpha>0, x>0,\\
\text{(d)} & \left\vert H\left( \alpha,x\right) \right\vert \leq
H_{2}\left( \alpha,x\right) , & \alpha>0, x>0,\\
\text{(e)} & C\left( \alpha\right) =\alpha^{\alpha+1}
{\displaystyle\int_{0}^{\infty}}
\dfrac{t^{\alpha}}{\sinh\left( \alpha t\right) }dt, & \alpha>0,\\
\text{(f)} & C\left( \alpha-1\right) =\alpha^{\alpha}
{\displaystyle\int_{0}^{\infty}}
\dfrac{t^{\alpha-1}}{\sinh\left( \alpha t\right) }dt, & \alpha>1.
\end{array}
\label{Properties1}
\end{equation}

Note that (\ref{Properties1}f) is not an easy consequence of
(\ref{Properties1}e). We also remark, that for $\alpha\geq1$ equation
(\ref{Properties1}a) remains also valid for $x=0$, by interpreting both sides
as their $\lim_{x\rightarrow0^{+}}$. The same holds true for
(\ref{Properties1}b) and (\ref{Properties1}d) for $\alpha>0$. We then have%
\begin{align}
H_{1}\left( \alpha,0\right)  & =\left\{
\begin{array}
[c]{ll}%
\frac{\pi}{2}, & \alpha=1,\\
0, & \alpha>1,
\end{array}
\right. \nonumber\\
H\left( \alpha,0\right)  & =H_{2}\left( \alpha,0\right) =0,\quad\alpha>0.
\label{limitproperty}%
\end{align}
Then, using (\ref{limitproperty}), we define
\begin{equation}
H_{\alpha}\left( x\right) =\left\vert x\right\vert ^{\alpha}-\dfrac{2}{\pi
}\sin\dfrac{\pi\alpha}{2}H\left( \alpha,x\right) ,\quad\alpha>0,x\geq0.
\label{Entire4}
\end{equation}

Next, we record
\begin{equation}
\begin{array}
[c]{lll}
\text{(a)} &
{\displaystyle\int_{0}^{c}}
x^{\alpha-1}e^{-\alpha x}\left( 1-x\right) dx & \\
& =
{\displaystyle\int_{c}^{\infty}}
x^{\alpha-1}e^{-\alpha x}\left( x-1\right) dx=\dfrac{c^{\alpha}e^{-\alpha
c}}{\alpha}, & \alpha>0,c\geq0,\\
\text{(b)} &
{\displaystyle\int_{0}^{\infty}}
x^{\alpha-2}e^{-\alpha x}dx=\dfrac{\Gamma\left( \alpha-1\right) }%
{\alpha^{\alpha-1}}, & \alpha>1,\\
\text{(c)} &
{\displaystyle\int_{0}^{\infty}}
x^{\alpha-1}e^{-\alpha x}dx=
{\displaystyle\int_{0}^{\infty}}
x^{\alpha}e^{-\alpha x}dx=\dfrac{\Gamma\left( \alpha\right) }{\alpha
^{\alpha}}, & \alpha>0,\\
\text{(d)} & \Gamma\left( \alpha\right) >\sqrt{\dfrac{2\pi}{\alpha}}\left(
\dfrac{\alpha}{e}\right) ^{\alpha}, & \alpha\geq1.
\end{array}
\label{Properties2}%
\end{equation}

\begin{proof}
Both equations in (\ref{Properties2}c) as well as (\ref{Properties2}b) are
derived directly from (\cite{Gradshteyn}, 3.381.4). The equations
(\ref{Properties2}a) are then an easy consequence of (\ref{Properties2}c)
combined together with (\cite{Gradshteyn}, 3.381.3 and 8.356.2). Inequality
(\ref{Properties2}d) can be derived from (\cite{Gradshteyn}, 8.327).
\end{proof}

\noindent
Finally, we apologize for the repulsive notation $\left\Vert f\left( x \right) \right\Vert$ instead of $\left\Vert f \right\Vert$ that we occasionally use in this paper.

\section{The limiting error term}

Let $\alpha>0$ and $n\in\mathbb{N}$. We recall the definition of the nodes
$x_{j}^{\left( 2n+1\right) }=\cos\frac{\left( j-1/2\right) \pi}{2n+1}$ for
$j=1,2,\ldots2n+1$ to be the zeros of the Chebyshev polynomial $T_{2n+1}$.
Further denote by $P_{2n}^{\left( 2\right) }$ the unique Lagrange
interpolation polynomial for $\left\vert x\right\vert ^{\alpha}$ in the
interval $\left[ -1,1\right] $.

\medskip\noindent Then, for $2n>\alpha>0$ and all $x\in\left[ -1,1\right] $, 
we simply derive from (\cite{Revers1}, Theorem 1) the asymptotic formula%
\[
\left( 2n\right) ^{\alpha}\left( \left\vert x\right\vert ^{\alpha}%
-P_{2n}^{\left( 2\right) }\left( x\right) \right) =\left( -1\right)
^{n}\frac{2}{\pi}\sin\frac{\pi\alpha}{2}\left( 1-\frac{1}{2n+1}\right)
\]%
\begin{equation}
\cdot T_{2n+1}\left( x\right) \int_{0}^{\infty}\frac{t^{\alpha}}%
{\sinh\left( t\right) }\frac{2nx}{\left( 2nx\right) ^{2}+t^{2}}dt+o\left(
1\right) ,\quad n\rightarrow\infty, \label{mikeasymptotics}%
\end{equation}
where o$\left( 1\right) $ is independent of $x$.

\medskip\noindent
The objective now is to find its limiting error term in the
$L_{\infty}$ norm. Since the error term is symmetric in $\left[ -1,1\right]
$ we prove the following

\begin{theorem}
\label{Theorem31}
Let $\alpha>0.$ Then we have%
\begin{align*}
\lim_{n\rightarrow\infty}\left( 2n\right) ^{\alpha}\left\Vert \left\vert
x\right\vert ^{\alpha}-P_{2n}^{\left( 2\right) }\right\Vert _{L_{\infty
}\left[ 0,1\right] } & =\frac{2}{\pi}\left\vert \sin\frac{\pi\alpha}%
{2}\right\vert \left\Vert H\left( \alpha,\cdot\right) \right\Vert
_{L_{\infty}\left[ 0,\infty\right) }\\
& =\frac{2}{\pi}\left\vert \sin\frac{\pi\alpha}{2}\right\vert \sup
_{x\in\left[ 0,\infty\right) }%
%TCIMACRO{\dint _{0}^{\infty}}%
%BeginExpansion
{\displaystyle\int_{0}^{\infty}}
%EndExpansion
\dfrac{t^{\alpha}}{\sinh\left( t\right) }\dfrac{x\left\vert \sin
x\right\vert }{x^{2}+t^{2}}dt.
\end{align*}
\end{theorem}

\begin{theorem}
\label{Theorem32}
Let $\alpha>0.$ Then, uniformly on compact subsets in
$\left[ 0,\infty\right) $,%
\[
\lim_{n\rightarrow\infty}\left( 2n\right) ^{\alpha}P_{2n}^{\left( 2\right)
}\left( \frac{x}{2n}\right) =H_{\alpha}\left( x\right) .
\]
\end{theorem}

\begin{theorem}
\label{Theorem33}
Let $\alpha>0$ be not an even integer. Then $H_{\alpha}$ (interpreted as its
extension into the complex domain) is an entire function of exponential type
1, interpolating $\left\vert x\right\vert ^{\alpha}$ at the interpolation
points $\left\{ k\pi:k=0,1,2,\ldots\right\} $ and $H_{\alpha}$ admits a
representation as an interpolating series of the following form. Denote by
$N=\left[ \alpha/2\right] .$ Then, for all $x\in\mathbb{R}$, we have
\begin{align}
H_{\alpha}\left( x\right)  & =\sin x\left( \frac{2}{\pi}\sum_{n=0}%
^{N-1}\sin\left( \frac{\pi\left( \alpha-2n-2\right) }{2}\right) C\left(
\alpha-2n-2\right) x^{2n+1}\right. \nonumber\\
& \left. +2x^{2N+1}\sum_{k=1}^{\infty}\left( -1\right) ^{k}\frac{\left(
k\pi\right) ^{\alpha-2N}}{x^{2}-\left( k\pi\right) ^{2}}\right) .
\label{Entire5}%
\end{align}
For the special case $0<\alpha<2$ the expansion is then represented by%
\begin{align}
H_{\alpha}\left( x\right) =2x\sin x\sum_{k=1}^{\infty}\left( -1\right)
^{k}\frac{\left( k\pi\right) ^{\alpha}}{x^{2}-\left( k\pi\right) ^{2}}.
\label{Entire6}%
\end{align}
\end{theorem}

\begin{figure}[th]
\begin{center}
\includegraphics[width=0.6\textwidth]{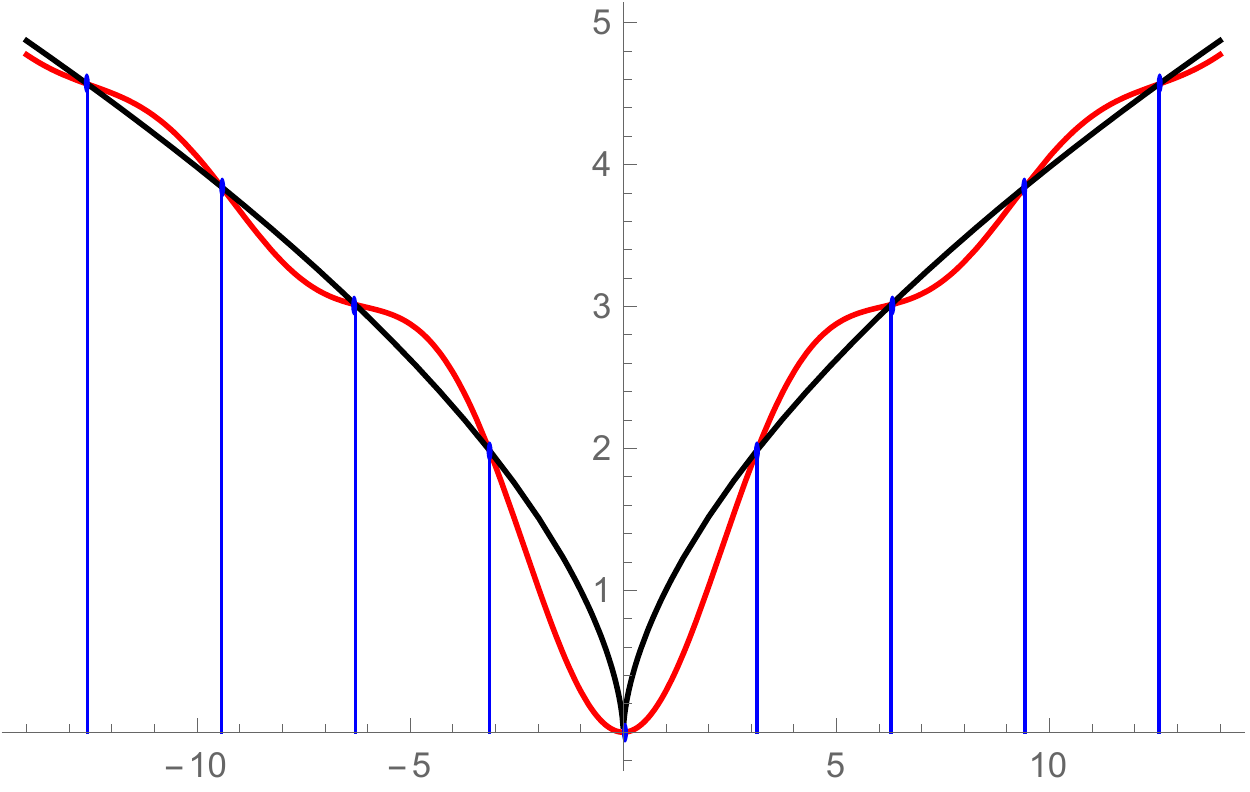}
\end{center}
\caption{Interpolating entire function $H_{\alpha}$ of exponential type 1 from
(\ref{Entire6})}%
\end{figure}

\noindent We start with the proof for Theorem \ref{Theorem31} by splitting it
in several Lemmas. First, we present without a proof the following two Lemmas.

\begin{lemma}
\label{lemma31}Let $x\in\left[ 0,\frac{1}{2}\right] $. Then $\left\vert
\arcsin x-x\right\vert \leq x^{2}.$
\end{lemma}

\begin{lemma}
\label{lemma32}For $n\in\mathbb{N}$ and $x\in\left[ -1,1\right] $ we have%
\[
T_{2n+1}\left( x\right) =\left( -1\right) ^{n}\sin\left( \left(
2n+1\right) \arcsin x\right) .
\]
\end{lemma}

\begin{lemma}
\label{lemma33}Let $n\in\mathbb{N}$ and $x\in\left[ -2n,2n\right] .$ Then we have%
\[
\left\vert \frac{T_{2n+1}\left( \frac{x}{2n}\right) }{x}\right\vert
\leq1+\frac{1}{2n}.
\]
\end{lemma}

\begin{proof}
The assertion is an easy consequence from (\cite{Revers1}, Lemma 10).
\end{proof}

\begin{lemma}
\label{lemma34}Let $C>0$ be fixed, $\varepsilon>0$ and $n>\max\left(
C,\frac{C}{\varepsilon}\right) $. Then
\[
\left\Vert \frac{T_{2n+1}\left( \frac{x}{2n}\right) }{x}-\frac{\left(
-1\right) ^{n}\sin\left( \left( 2n+1\right) \frac{x}{2n}\right) }%
{x}\right\Vert _{L_{\infty}\left[ 0,C\right] }<\varepsilon.
\]
\end{lemma}

\begin{proof}
For $x\in\left[ 0,C\right] $ we get $0\leq\frac{x}{2n}\leq\frac{C}{2n}%
<\frac{C}{2C}=\frac{1}{2}$. Then, using Lemma \ref{lemma31} and Lemma \ref{lemma32}, we estimate
\begin{align*}
& \left\vert \frac{T_{2n+1}\left( \frac{x}{2n}\right) }{x}-\frac{\left(
-1\right) ^{n}\sin\left( \left( 2n+1\right) \frac{x}{2n}\right) }%
{x}\right\vert \\
& = \frac{1}{x}\left\vert \sin\left( \left( 2n+1\right) \arcsin\frac{x}%
{2n}\right) -\sin\left( \left( 2n+1\right) \frac{x}{2n}\right)
\right\vert \\
& \leq \frac{2n+1}{x}\left\vert \arcsin\frac{x}{2n}-\frac{x}{2n}\right\vert \leq\frac{2n+1}{x}\left( \frac{x}{2n}\right) ^{2}\leq\frac{C}%
{n}<\varepsilon.
\end{align*}
\end{proof}

\begin{lemma}
\label{lemma35}Let $C>0$ be fixed, $\varepsilon>0$ and $n>\frac{1}%
{2\varepsilon}$. Then%
\[
\left\Vert \frac{\sin\left( \left( 2n+1\right) \frac{x}{2n}\right) }%
{x}-\frac{\sin x}{x}\right\Vert _{L_{\infty}\left[ 0,C\right] }%
<\varepsilon.
\]
\end{lemma}

\begin{proof}
Let $x\in\left[ 0,C\right]$. Then by a standard argument we arrive at
\[
\left\vert \frac{\sin\left( \left( 2n+1\right) \frac{x}{2n}\right) }%
{x}-\frac{\sin x}{x}\right\vert \leq\frac{1}{x}\left\vert \left( 2n+1\right)
\frac{x}{2n}-x\right\vert = \frac{1}{2n}<\varepsilon.
\]
\end{proof}

\begin{lemma}
\label{lemma36}Let $C>0$ be fixed, $\varepsilon>0$ and $n>\max\left(
C,\frac{C}{\varepsilon},\frac{1}{2\varepsilon}\right) $. Then%
\[
\left\Vert \frac{T_{2n+1}\left( \frac{x}{2n}\right) }{x}-\left( -1\right)
^{n}\frac{\sin x}{x}\right\Vert _{L_{\infty}\left[ 0,C\right] }%
<2\varepsilon.
\]
\end{lemma}

\begin{proof}
This follows directly by applying the triangle inequality combined together with Lemma
\ref{lemma34} and Lemma \ref{lemma35}.
\end{proof}

\begin{lemma}
\label{lemma37}Let $C>0$ be fixed, $\varepsilon>0$ and $n>\max\left(
C,\frac{C}{\varepsilon},\frac{1}{2\varepsilon}\right) .$ Then, for $\alpha
>0$, we have
\[
\left\Vert T_{2n+1}\left( \frac{x}{2n}\right) H_{1}\left( \alpha,x\right)
\right\Vert _{L_{\infty}\left[ 0,C\right] }\leq\left\Vert H\left(
\alpha,x\right) \right\Vert _{L_{\infty}\left[ 0,\infty\right)
}+2\varepsilon\cdot C\left( \alpha\right) .
\]
\end{lemma}

\begin{proof}
First, we remark that for $\alpha>0$ the left-hand side in Lemma \ref{lemma37} is well defined by applying (\ref{limitproperty}) together with Lemma \ref{lemma33}. Using again the triangle inequality together with Lemma \ref{lemma36} and formula (\ref{Properties1}c), we arrive at
\begin{align*}
&\left\Vert T_{2n+1}\left( \frac{x}{2n}\right) H_{1}\left( \alpha,x\right)
\right\Vert _{L_{\infty}\left[ 0,C\right] } 
= \left\Vert \frac
{T_{2n+1}\left( \frac{x}{2n}\right) }{x}H_{2}\left( \alpha,x\right) \right\Vert
_{L_{\infty}\left[ 0,C\right] }\\
& \leq \left\Vert \frac{T_{2n+1}\left( \frac{x}{2n}\right) }{x}-\left(
-1\right) ^{n}\frac{\sin x}{x}\right\Vert _{L_{\infty}\left[ 0,C\right]
}\left\Vert H_{2}\left( \alpha,x\right) \right\Vert _{L_{\infty}\left[
0,C\right] } \\
& + \left\Vert \sin x\cdot H_{1}\left( \alpha,x\right) \right\Vert
_{L_{\infty}\left[ 0,C\right] } \leq 2\varepsilon C\left( \alpha\right) +\left\Vert H\left( \alpha,x\right)
\right\Vert _{L_{\infty}\left[ 0,\infty\right) }.
\end{align*}
\end{proof}

\noindent
Our first substantial result is now the following

\begin{lemma}
\label{lemma38}Let $\alpha>0$. Then
\[
\overline{\lim}_{n\rightarrow\infty}\left\Vert T_{2n+1}\left( \frac{x}%
{2n}\right) H_{1}\left( \alpha,x\right) \right\Vert _{L_{\infty}\left[
0,2n\right] }\leq\left\Vert H\left( \alpha,x\right) \right\Vert
_{L_{\infty}\left[ 0,\infty\right) }.
\]
\end{lemma}

\begin{proof}
Let $\varepsilon>0,C>\frac{C\left( \alpha\right) }{\varepsilon}$ and
$n>\max\left( C,\frac{C}{\varepsilon},\frac{1}{2\varepsilon}\right)$. Then%
\begin{align*}
&\left\Vert T_{2n+1}\left( \frac{x}{2n}\right) H_{1}\left( \alpha,x\right)
\right\Vert _{L_{\infty}\left[ 0,2n\right] } \\
& \leq \left\Vert
T_{2n+1}\left( \frac{x}{2n}\right) H_{1}\left( \alpha,x\right) \right\Vert
_{L_{\infty}\left[ 0,C\right] } + \left\Vert T_{2n+1}\left( \frac{x}{2n}\right) H_{1}\left(
\alpha,x\right) \right\Vert _{L_{\infty}\left[ C,2n\right] }.
\end{align*}
Using (\ref{Properties1}c), the latter part can be estimated to
\begin{align*}
\left\Vert T_{2n+1}\left( \frac{x}{2n}\right) H_{1}\left( \alpha,x\right)
\right\Vert _{L_{\infty}\left[ C,2n\right] } & =\left\Vert \frac
{T_{2n+1}\left( \frac{x}{2n}\right) }{x}H_{2}\left( \alpha,x\right)
\right\Vert _{L_{\infty}\left[ C,2n\right] }\\
& \leq\frac{1}{C}\cdot C\left( \alpha\right) <\varepsilon.
\end{align*}
Combined together with the previous estimate and Lemma \ref{lemma37}, we
finally get
\[
\left\Vert T_{2n+1}\left( \frac{x}{2n}\right) H_{1}\left( \alpha,x\right)
\right\Vert _{L_{\infty}\left[ 0,2n\right] }\leq\left\Vert H\left(
\alpha,x\right) \right\Vert _{L_{\infty}\left[ 0,\infty\right)
}+2\varepsilon\cdot C\left( \alpha\right) +\varepsilon.
\]
By taking the $\overline{\lim}$ the result follows.
\end{proof}

\noindent
Now, we are turning to the $\underline{\lim}$ case.

\begin{lemma}
\label{lemma39}Let $\alpha>0$ and $C>0$ be fixed. Then
\[
\underline{\lim}_{n\rightarrow\infty}\left\Vert T_{2n+1}\left( \frac{x}%
{2n}\right) H_{1}\left( \alpha,x\right) \right\Vert _{L_{\infty}\left[
0,2n\right] }\geq\left\Vert H\left( \alpha,x\right) \right\Vert
_{L_{\infty}\left[ 0,C\right] }.
\]
\end{lemma}

\begin{proof}
Let $C>0$, $\varepsilon>0$ and $n>\max\left( C,\frac
{C}{\varepsilon},\frac{1}{2\varepsilon}\right) .$ Then, by applying again the
triangle inequality and combining together with Lemma \ref{lemma36} and
(\ref{Properties1}c), we estimate
\begin{align*}
& \left\Vert T_{2n+1}\left( \frac{x}{2n}\right) H_{1}\left( \alpha
,x\right) \right\Vert _{L_{\infty}\left[ 0,2n\right] }\\
& \geq \left\Vert T_{2n+1}\left( \frac{x}{2n}\right) H_{1}\left(
\alpha,x\right) \right\Vert _{L_{\infty}\left[ 0,C\right] }\\
& \geq \left\Vert H\left( \alpha,x\right) \right\Vert _{L_{\infty}\left[
0,C\right] }-\left\Vert \left( \frac{T_{2n+1}\left( \frac{x}{2n}\right)
}{x}-\frac{\left( -1\right) ^{n}\sin x}{x}\right) H_{2}\left(
\alpha,x\right) \right\Vert _{L_{\infty}\left[ 0,C\right] }\\
& \geq \left\Vert H\left( \alpha,x\right) \right\Vert _{L_{\infty}\left[
0,C\right] }-2\varepsilon\left\Vert H_{2}\left( \alpha,x\right) \right\Vert
_{L_{\infty}\left[ 0,C\right] }\\
& \geq \left\Vert H\left( \alpha,x\right) \right\Vert _{L_{\infty}\left[
0,C\right] }-2\varepsilon\cdot C\left( \alpha\right) .
\end{align*}
Now, by taking $\underline{\lim}$ we establish the result.
\end{proof}

\noindent
Our second substantial result is the following

\begin{lemma}
\label{lemma310}Let $\alpha>0.$ Then
\[
\underline{\lim}_{n\rightarrow\infty}\left\Vert T_{2n+1}\left( \frac{x}%
{2n}\right) H_{1}\left( \alpha,x\right) \right\Vert _{L_{\infty}\left[
0,2n\right] }\geq\left\Vert H\left( \alpha,x\right) \right\Vert
_{L_{\infty}\left[ 0,\infty\right) }.
\]
\end{lemma}

\begin{proof}Let $\varepsilon>0$ and $C>\frac{C\left( \alpha\right) }{\varepsilon
}$. Then, starting with the right-hand side in Lemma \ref{lemma310}, we
estimate
\[
\left\Vert H\left( \alpha,x\right) \right\Vert _{L_{\infty}\left[
0,\infty\right) }\leq\left\Vert H\left( \alpha,x\right) \right\Vert
_{L_{\infty}\left[ 0,C\right] }+\left\Vert H\left( \alpha,x\right)
\right\Vert _{L_{\infty}\left[ C,\infty\right) }.
\]
Using again (\ref{Properties1}c), the latter part can be estimated to
\begin{align*}
\left\Vert H\left( \alpha,x\right) \right\Vert _{L_{\infty}\left[
C,\infty\right) } & =\left\Vert \frac{\sin x}{x}H_{2}\left( \alpha
,x\right) \right\Vert _{L_{\infty}\left[ C,\infty\right) }\\
& \leq\frac{1}{C}\cdot C\left( \alpha\right) <\varepsilon.
\end{align*}
Combined together with Lemma \ref{lemma39} and the previous estimate, we
arrive at
\begin{align*}
\left\Vert H\left( \alpha,x\right) \right\Vert _{L_{\infty}\left[
0,\infty\right) }-\varepsilon & \leq\left\Vert H\left( \alpha,x\right)
\right\Vert _{L_{\infty}\left[ 0,C\right] }\\
& \leq\underline{\lim}_{n\rightarrow\infty}\left\Vert T_{2n+1}\left( \frac
{x}{2n}\right) H_{1}\left( \alpha,x\right) \right\Vert _{L_{\infty}\left[
0,2n\right] }.
\end{align*}
Since the last expression holds for every $\varepsilon>0$ we establish the result.
\end{proof}

\begin{proof}[Proof of Theorem \ref{Theorem31}.] Let $\alpha>0$. Then%
\begin{align*}
& \left\Vert T_{2n+1}\left( x\right) \int_{0}^{\infty}\frac{t^{\alpha}%
}{\sinh\left( t\right) }\frac{2nx}{\left( 2nx\right) ^{2}+t^{2}%
}dt\right\Vert _{L_{\infty}\left[ 0,1\right] }\\
& = \left\Vert T_{2n+1}\left( \frac{x}{2n}\right) \int_{0}^{\infty}%
\frac{t^{\alpha}}{\sinh\left( t\right) }\frac{x}{x^{2}+t^{2}}dt\right\Vert
_{L_{\infty}\left[ 0,2n\right] }\\
& = \left\Vert T_{2n+1}\left( \frac{x}{2n}\right) H_{1}\left( \alpha
,x\right) \right\Vert _{L_{\infty}\left[ 0,2n\right] }.
\end{align*}
Combining now Lemma \ref{lemma38} and Lemma \ref{lemma310} together with (\ref{mikeasymptotics}), gives the result and we are finished.
\end{proof}

\begin{proof}[Proof of Theorem \ref{Theorem32}.]
Let $\alpha>0.$ From (\ref{mikeasymptotics}) it follows that for every
$\varepsilon>0$ we can find some $n_{0}=n_{0}\left( \varepsilon\right) $, 
such that for all $n>n_{0}$%
\begin{align*}
& \left\Vert \left( 2n\right) ^{\alpha}\left( \left\vert x\right\vert
^{\alpha}-P_{2n}^{\left( 2\right) }\left( x\right) \right) -\left(
-1\right) ^{n}\frac{2}{\pi}\sin\frac{\pi\alpha}{2}\left( 1-\frac{1}%
{2n+1}\right) \right. \\
& \cdot \left. T_{2n+1}\left( x\right) \int_{0}^{\infty}\frac{t^{\alpha}%
}{\sinh t}\frac{2nx}{\left( 2nx\right) ^{2}+t^{2}}dt\right\Vert _{L_{\infty
}\left[ 0,1\right] } <\varepsilon.
\end{align*}
Let $C>0$ be fixed, $\varepsilon>0$ and $n>\max\left( C,\frac{C}{\varepsilon
},\frac{1}{2\varepsilon},\frac{\alpha}{2},n_{0}\right) $. Then%
\begin{align}
& \left\Vert \left( 2n\right) ^{\alpha}P_{2n}^{\left( 2\right) }\left(
\frac{x}{2n}\right) -H_{\alpha}\left( x\right) \right\Vert _{L_{\infty
}\left[ 0,C\right] }\nonumber\\
& = \left\Vert \frac{2}{\pi}\sin\frac{\pi\alpha}{2}H\left( \alpha
,2nx\right) -\left( 2n\right) ^{\alpha}\left( \left\vert x\right\vert
^{\alpha}-P_{2n}^{\left( 2\right) }\left( x\right) \right) \right\Vert
_{L_{\infty}\left[ 0,\frac{C}{2n}\right] }\label{Estimate1}\\
& \leq \frac{2}{\pi}\left\vert \sin\frac{\pi\alpha}{2}\right\vert \left\Vert
H\left( \alpha,x\right) -\left( -1\right) ^{n}\frac{2n}{2n+1}%
T_{2n+1}\left( \frac{x}{2n}\right) H_{1}\left( \alpha,x\right) \right\Vert
_{L_{\infty}\left[ 0,C\right] }+\varepsilon.\nonumber
\end{align}
We proceed further by use of (\ref{Properties1}c), Lemma \ref{lemma33} and
Lemma \ref{lemma36}.%
\begin{align*}
& \left\Vert H\left( \alpha,x\right) -\left( -1\right) ^{n}\frac
{2n}{2n+1}T_{2n+1}\left( \frac{x}{2n}\right) H_{1}\left( \alpha,x\right)
\right\Vert _{L_{\infty}\left[ 0,C\right] }\\
& = \left\Vert H_{1}\left( \alpha,x\right) \left( \sin x-\left(
-1\right) ^{n}\frac{2n}{2n+1}T_{2n+1}\left( \frac{x}{2n}\right) \right)
\right\Vert _{L_{\infty}\left[ 0,C\right] }\\
& \leq C\left( \alpha\right) \left( \left\Vert \frac{T_{2n+1}\left(
\frac{x}{2n}\right) }{x}-\left( -1\right) ^{n}\frac{\sin x}{x}\right\Vert
_{L_{\infty}\left[ 0,C\right] }+\frac{1}{2n+1}\left\Vert \frac
{T_{2n+1}\left( \frac{x}{2n}\right) }{x}\right\Vert _{L_{\infty}\left[
0,C\right] }\right) \\
& \leq C\left( \alpha\right) \left( 2\varepsilon+\frac{1}{2n}\right) \\
& \leq C\left( \alpha\right) 3\varepsilon.
\end{align*}
Combining together with (\ref{Estimate1}), we obtain for every $\varepsilon>0$
and $n$ sufficiently large,
\[
\left\Vert \left( 2n\right) ^{\alpha}P_{2n}^{\left( 2\right) }\left(
\frac{\cdot}{2n}\right) -H_{\alpha}\right\Vert _{L_{\infty}\left[
0,C\right] }\leq\frac{2}{\pi}\left\vert \sin\frac{\pi\alpha}{2}\right\vert
C\left( \alpha\right) 3\varepsilon+\varepsilon.
\]
Since any compact set $K$ in $\left[ 0,\infty\right) $ can be included in
some interval $\left[ 0,C\right] $ the result is established.
\end{proof}

\begin{proof}[Proof of Theorem \ref{Theorem33}.]
The expansion of $H_{\alpha}$ into the interpolating series (\ref{Entire5})
follows after some routine arguments from (\cite{Ganzburg1}, Formula 4.14). The special case (\ref{Entire6}) can be directly seen from (\cite{Ganzburg1}, Formula 4.16). The fact that $H_{\alpha}$ is an entire function of exponential type $1$ can
now be deduced from (\cite{Timan}, p. 183, Formula 15). The interpolation property is an easy consequence of (\ref{Entire4}).
\end{proof}

\section{The Envelope function}

In this section we consider the envelope error function $H_{1}\left(
\alpha,\cdot\right) $ with respect to $\left\vert H\left( \alpha
,\cdot\right) \right\vert $. Our next objective is to establish an
asymptotics for $\left\Vert H_{1}\left( \alpha,\cdot\right) \right\Vert
_{L_{\infty}\left[ 0,\infty\right) }$ when $\alpha\rightarrow\infty$. We show

\begin{theorem}
\label{Theorem41}Let $\alpha\geq2.$ Then, we have%
\[
\frac{C\left( \alpha\right) }{1+2\alpha}\left( 1-\frac{1}{\sqrt{\alpha}%
}\right) \leq H_{1}\left( \alpha,\alpha\right) \leq\left\Vert H_{1}\left(
\alpha,\cdot\right) \right\Vert _{L_{\infty}\left[ 0,\infty\right) }%
\leq\frac{C\left( \alpha\right) }{1+2\alpha}\left( 1+\frac{2}{\sqrt{\alpha
}}\right) .
\]
\end{theorem}

\begin{figure}[th]
\begin{center}
\includegraphics[width=0.45\textwidth]{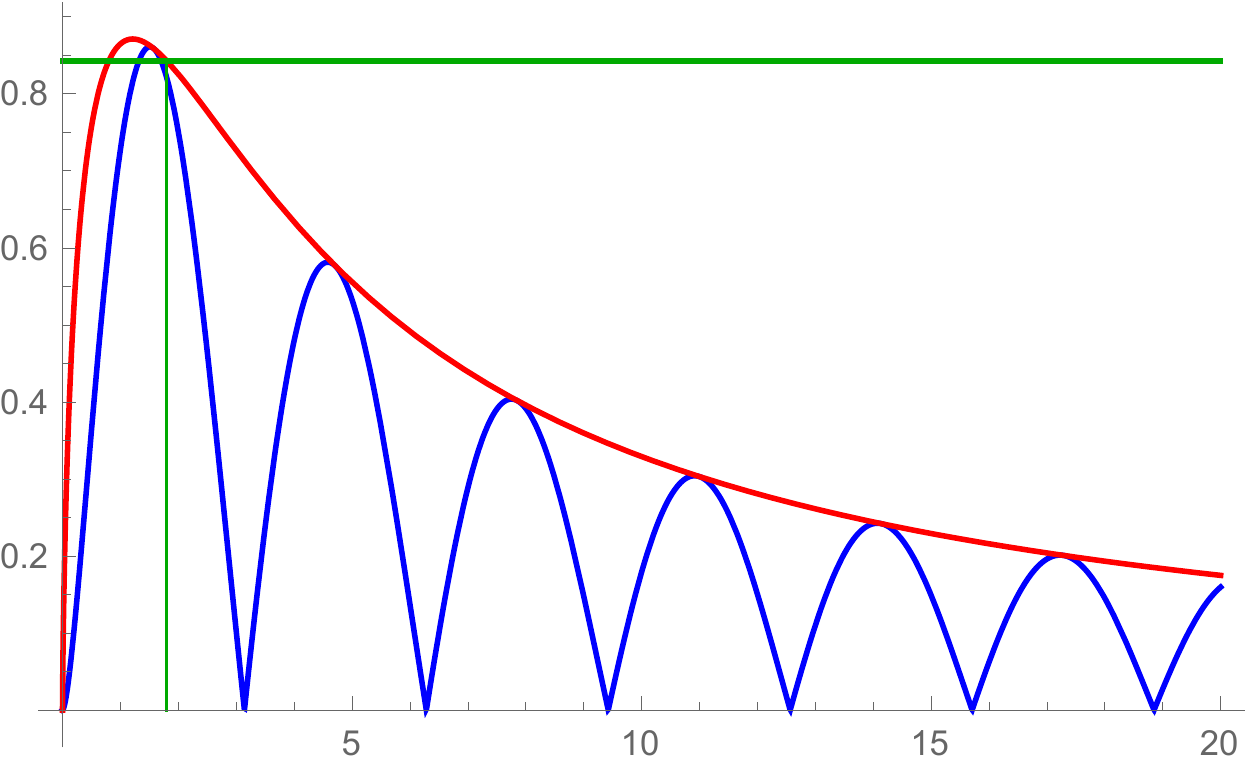} \hfill
\includegraphics[width=0.45\textwidth]{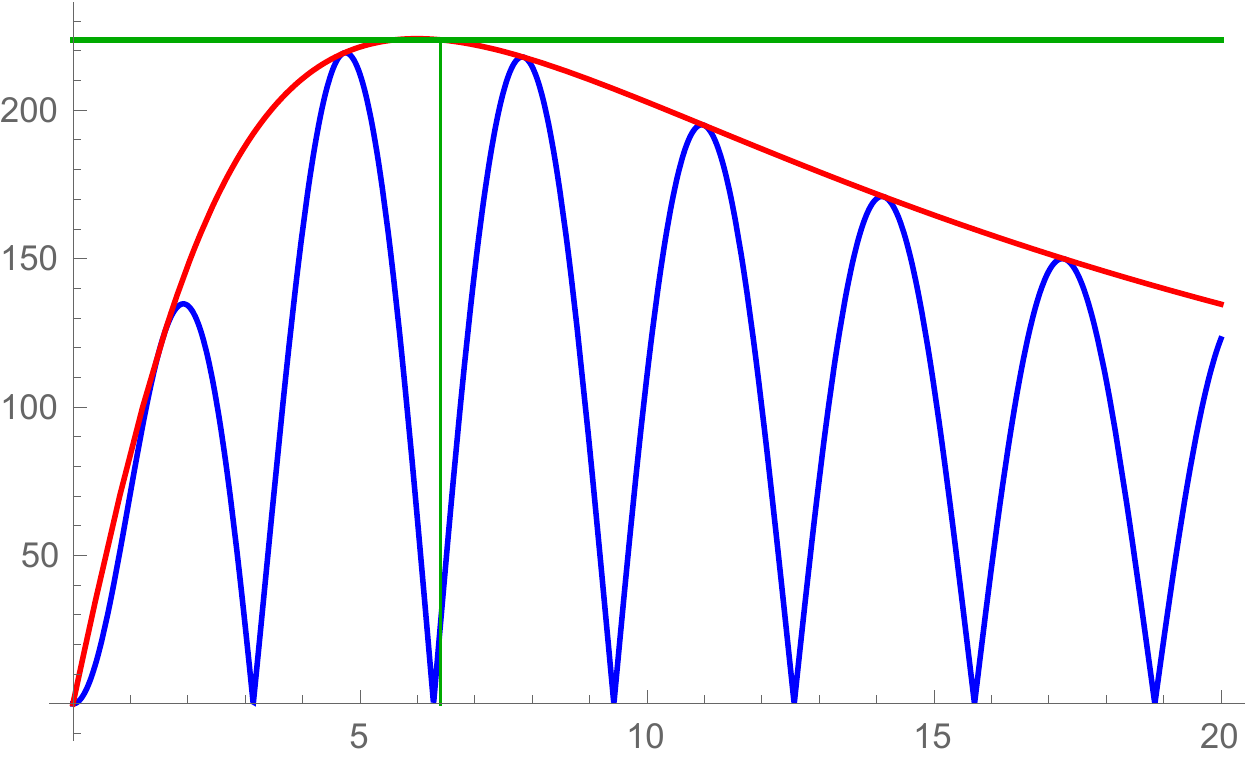}
\end{center}
\caption{The error function $\left\vert H\left( \alpha,\cdot\right)
\right\vert $, its envelope $H_{1}\left( \alpha,\cdot\right) $ and the point evaluation $H_{1}\left( \alpha,\alpha \right) $.}
\label{Figure2}
\end{figure}

\noindent
Figure \ref{Figure2} shows the functions $\left\vert
H\left( \alpha,\cdot\right) \right\vert $ and $H_{1}\left( \alpha
,\cdot\right) $ as well as their point evaluations for values $\alpha=1.8$
and $\alpha=6.4$. The figure suggests that a useful lower estimate for
$\left\Vert H_{1}\left( \alpha,\cdot\right) \right\Vert $ should be
derivable when determining its point evaluation, i.e. $H_{1}\left(
\alpha,\alpha\right) $, at least for large values for $\alpha$.

\medskip
\noindent
We start proving Theorem \ref{Theorem41} by splitting it in several
Lemmas. First, we present the following five Lemmas without proof. They can be
derived by some standard analysis arguments.

\begin{lemma}
\label{lemma41}The function $f\left( x\right) =\left( 1+\frac{1}{x}\right)
^{x}$ is monotonically increasing in $x\in\left( 0,\infty\right) $ and
$f\left( x\right) \leq e.$
\end{lemma}

\begin{lemma}
\label{lemma42}For $x>0$ we have%
\[
\frac{1}{1-e^{-x}}-\frac{1}{x}\leq1.
\]
\end{lemma}

\begin{lemma}
\label{lemma43}Let $\alpha>0.$ The function%
\[
f\left( x\right) =\frac{x}{1-e^{-2\alpha x}}%
\]
is convex for $x\geq0$. Here $f\left( 0\right) =\lim_{x\rightarrow0^{+}%
}f\left( x\right) =\frac{1}{2\alpha}$.
\end{lemma}

\begin{lemma}
\label{lemma44}Let $\alpha>0.$ Then, for $x\in\left[ 0,1+\frac{1}{2\alpha
}\right] $, we have%
\[
\frac{x}{1-e^{-2\alpha x}}\leq\left( \frac{1}{1-e^{-2\alpha-1}}-\frac
{1}{1+2\alpha}\right) x+\frac{1}{2\alpha}.
\]
\end{lemma}

\begin{lemma}
\label{lemma45}For $x\geq0$ denote by $f\left( x\right) =x\left(
x+1\right) /\left( x^{2}+1\right) $. Then, for $x\geq0$, we have%
\[
f\left( x\right) \leq f\left( 1+\sqrt{2}\right) =\frac{1+\sqrt{2}}{2}.
\]
\end{lemma}

\noindent Our first substantial result is now the following

\begin{lemma}
\label{lemma46}Let $\alpha\geq1.$ Then
\[
\frac{C\left( \alpha\right) }{1+2\alpha}\left( 1-\frac{1}{\sqrt{\alpha}%
}\right) \leq H_{1}\left( \alpha,\alpha\right) .
\]
\end{lemma}

\begin{proof}
By some routine arguments and using Lemma (\ref{lemma45}), (\ref{Properties2}%
a), (\ref{Properties2}d) and (\ref{Properties2}c), we estimate
\begin{align*}
& \frac{1+2\alpha}{\alpha}\int_{0}^{\infty}\frac{x^{\alpha}}{\sinh\alpha
x}\left( \frac{\alpha}{1+2\alpha}-\frac{1}{1+x^{2}}\right) dx\\
& = \int_{0}^{\infty}\frac{x^{\alpha}}{\sinh\alpha x}\frac{x^{2}-1}{x^{2}%
+1}dx-\frac{1}{\alpha}\int_{0}^{\infty}\frac{x^{\alpha}}{\sinh\alpha x}%
\frac{1}{x^{2}+1}dx\\
& \leq \int_{1}^{\infty}\frac{x^{\alpha}}{\sinh\alpha x}\frac{x^{2}-1}%
{x^{2}+1}dx\\
&= 2\int_{1}^{\infty}x^{\alpha-1}e^{-\alpha x}\frac{1}{1-e^{-2\alpha x}}%
\frac{x\left( x-1\right) \left( x+1\right) }{x^{2}+1}dx\\
&\leq \frac{1+\sqrt{2}}{1-e^{-2}}\int_{1}^{\infty}x^{\alpha-1}e^{-\alpha
x}\left( x-1\right) dx\\
&\leq \frac{2}{\sqrt{\alpha}}\frac{1}{\alpha^{\alpha}}\sqrt{\frac{2\pi
}{\alpha}}\left( \frac{\alpha}{e}\right) ^{\alpha}\\
&< \frac{2}{\sqrt{\alpha}}\frac{1}{\alpha^{\alpha}}\Gamma\left(
\alpha\right) \\
&= \frac{2}{\sqrt{\alpha}}\int_{0}^{\infty}x^{\alpha}e^{-\alpha x}dx\\
&\leq \frac{1}{\sqrt{\alpha}}\int_{0}^{\infty}\frac{x^{\alpha}}{\sinh\alpha
x}dx.
\end{align*}
We summarize
\begin{align}
& \int_{0}^{\infty}\frac{x^{\alpha}}{\sinh\alpha x}\left( \frac{1}{1+x^{2}%
}-\frac{\alpha}{1+2\alpha}\right) dx\nonumber\\
& \geq-\frac{1}{\sqrt{\alpha}}\frac{\alpha}{1+2\alpha}\int_{0}^{\infty}%
\frac{x^{\alpha}}{\sinh\alpha x}dx.\label{Asy1}%
\end{align}
Now, using (\ref{Properties1}a), (\ref{Properties1}e) together with
(\ref{Asy1}), we obtain the final result
\begin{align*}
H_{1}\left( \alpha,\alpha\right)   & =\alpha^{\alpha}F\left( \alpha
,\alpha\right) \\
& =\frac{C\left( \alpha\right) }{1+2\alpha}+\alpha^{\alpha}\int_{0}^{\infty
}\frac{t^{\alpha}}{\sinh\alpha t}\left( \frac{1}{1+t^{2}}-\frac{\alpha
}{1+2\alpha}\right) dt\\
& \geq\frac{C\left( \alpha\right) }{1+2\alpha}-\frac{\alpha^{\alpha}}%
{\sqrt{\alpha}}\frac{\alpha}{1+2\alpha}\int_{0}^{\infty}\frac{t^{\alpha}%
}{\sinh\alpha t}dt\\
& =\frac{C\left( \alpha\right) }{1+2\alpha}\left( 1-\frac{1}{\sqrt{\alpha}%
}\right) .
\end{align*}
\end{proof}

\noindent Next, we show
\begin{lemma}
\label{lemma47}Let $\alpha>1.$ Then
\[
\left\Vert H_{1}\left( \alpha,\cdot\right) \right\Vert _{L_{\infty}\left[
0,\infty\right) }\leq\frac{1}{2}C\left( \alpha-1\right) .
\]
\end{lemma}

\begin{proof}From (\ref{limitproperty}) it follows that we can restrict ourselves to
values $H_{1}\left( \alpha,x\right) $ for $x>0.$ Thus
\begin{align*}
\left\Vert H_{1}\left( \alpha,\cdot\right) \right\Vert _{L_{\infty}\left[
0,\infty\right) } & = \left\Vert \int_{0}^{\infty}\frac{t^{\alpha}}{\sinh t}
\frac{x}{x^{2}+t^{2} }dt \right\Vert _{L_{\infty}\left( 0,\infty\right) }\\
& \leq\left\Vert \int_{0}^{\infty}\frac{t^{\alpha}}{\sinh t}\frac{x}%
{2xt}dt\right\Vert _{L_{\infty}\left( 0,\infty\right) }\\
& =\frac{1}{2}\int_{0}^{\infty}\frac{t^{\alpha-1}}{\sinh t}dt \\
& =\frac{1}{2}C\left( \alpha-1\right) .
\end{align*}
\end{proof}

\begin{lemma}
\label{lemma48}Let $\alpha\geq2.$ Then
\[
\frac{1}{2}C\left( \alpha-1\right) \leq\frac{C\left( \alpha\right)
}{1+2\alpha}\left( 1+\frac{2}{\sqrt{\alpha}}\right) .
\]
\end{lemma}

\begin{proof}
By using Lemma \ref{lemma44} and Lemma \ref{lemma42}, we begin with
\begin{align*}
& \int_{0}^{\infty}\frac{x^{\alpha-1}}{\sinh\alpha x}\left( 1+\frac
{1}{2\alpha}-x\right) dx\\
& \leq2\int_{0}^{1+\frac{1}{2\alpha}}x^{\alpha-2}e^{-\alpha x}\frac
{x}{1-e^{-2\alpha x}}\left( 1+\frac{1}{2\alpha}-x\right) dx\\
& \leq2\int_{0}^{1+\frac{1}{2\alpha}}x^{\alpha-2}e^{-\alpha x}\left( \left(
\frac{1}{1-e^{-2\alpha-1}}-\frac{1}{1+2\alpha}\right) x+\frac{1}{2\alpha
}\right) \left( 1+\frac{1}{2\alpha}-x\right) dx\\
& \leq2\int_{0}^{1+\frac{1}{2\alpha}}x^{\alpha-2}e^{-\alpha x}\left(
x+\frac{1}{2\alpha}\right) \left( 1+\frac{1}{2\alpha}-x\right) dx\\
& =2\int_{0}^{1+\frac{1}{2\alpha}}x^{\alpha-2}e^{-\alpha x}\left(
x-x^{2}+\frac{1}{2\alpha}+\frac{1}{4\alpha^{2}}\right) dx.
\end{align*}
Note, that for $\alpha\geq\frac{1}{2}$ we have $1/\alpha\geq1/\left(
2\alpha\right) +1/\left( 4\alpha^{2}\right) .$ From this, by using
(\ref{Properties2}a), it follows that
\begin{align*}
& 2\int_{0}^{1+\frac{1}{2\alpha}}x^{\alpha-2}e^{-\alpha x}\left(
x-x^{2}+\frac{1}{2\alpha}+\frac{1}{4\alpha^{2}}\right) dx\\
& \leq2\int_{0}^{1+\frac{1}{2\alpha}}x^{\alpha-1}e^{-\alpha x}\left(
1-x\right) dx+\frac{2}{\alpha}\int_{0}^{1+\frac{1}{2\alpha}}x^{\alpha
-2}e^{-\alpha x}dx\\
& =\frac{2}{\alpha}\sqrt{\left( 1+\frac{1}{2\alpha}\right) ^{2\alpha}%
}e^{-\alpha}e^{-\frac{1}{2}}+\frac{2}{\alpha}\int_{0}^{1+\frac{1}{2\alpha}%
}x^{\alpha-2}e^{-\alpha x}dx.
\end{align*}
Then, using Lemma \ref{lemma41} and (\ref{Properties2}b), we can further
estimate to
\begin{align*}
& \frac{2}{\alpha}\sqrt{\left( 1+\frac{1}{2\alpha}\right) ^{2\alpha}%
}e^{-\alpha}e^{-\frac{1}{2}}+\frac{2}{\alpha}\int_{0}^{1+\frac{1}{2\alpha}%
}x^{\alpha-2}e^{-\alpha x}dx\\
& \leq\frac{2}{\alpha}e^{\frac{1}{2}}e^{-\alpha}e^{-\frac{1}{2}}+\frac
{2}{\alpha}\int_{0}^{\infty}x^{\alpha-2}e^{-\alpha x}dx\\
& =\frac{2}{\alpha}e^{-\alpha}+\frac{2}{\alpha}\frac{\Gamma\left(
\alpha-1\right) }{\alpha^{\alpha-1}}\\
& =\frac{2}{\alpha}e^{-\alpha}+\frac{2}{\alpha-1}\frac{1}{\alpha^{\alpha}%
}\Gamma\left( \alpha\right) .
\end{align*}
We collect for $\alpha\geq2$ the inequality $2/\alpha\geq1/\left( \alpha-1\right) $.
Now, using (\ref{Properties2}d) and (\ref{Properties2}c), we estimate further
\begin{align*}
& \frac{2}{\alpha}e^{-\alpha}+\frac{2}{\alpha-1}\frac{1}{\alpha^{\alpha}%
}\Gamma\left( \alpha\right) \\
& \leq\frac{1}{\sqrt{\alpha}}\frac{\Gamma\left( \alpha\right) }%
{\alpha^{\alpha}}\left( \frac{2}{\sqrt{2\pi}}+\frac{4}{\sqrt{\alpha}}\right)
\\
& \leq\frac{4}{\sqrt{\alpha}}\frac{\Gamma\left( \alpha\right) }%
{\alpha^{\alpha}}=\frac{4}{\sqrt{\alpha}}\int_{0}^{\infty}x^{\alpha}e^{-\alpha
x}dx\\
& \leq\frac{2}{\sqrt{\alpha}}\int_{0}^{\alpha}\frac{x^{\alpha}}{\sinh\alpha
x}dx.
\end{align*}
Combining all together, we obtain for all $\alpha\geq2$,
\[
\int_{0}^{\alpha}\frac{t^{\alpha-1}}{\sinh\alpha t}dt\leq\frac{2\alpha
}{1+2\alpha}\left( 1+\frac{2}{\sqrt{\alpha}}\right) \int_{0}^{\alpha}
\frac{t^{\alpha}}{\sinh\alpha t}dt.
\]
Finally, using (\ref{Properties1}f) and (\ref{Properties1}e), we arrive at
\begin{align*}
\frac{1}{2}C\left( \alpha-1\right)  & =\frac{\alpha^{\alpha}}{2}\int
_{0}^{\alpha}\frac{t^{\alpha-1}}{\sinh\alpha t}dt\\
& \leq\frac{\alpha^{\alpha}}{2}\frac{2\alpha}{1+2\alpha}\left( 1+\frac
{2}{\sqrt{\alpha}}\right) \int_{0}^{\alpha}\frac{t^{\alpha}}{\sinh\alpha
t}dt\\
& =\left( 1+\frac{2}{\sqrt{\alpha}}\right) \frac{C\left( \alpha\right)
}{1+2\alpha}.
\end{align*}
\end{proof}

\begin{proof}[Proof of Theorem \ref{Theorem41}.] The Theorem is now an easy
consequence of Lemma \ref{lemma46}, Lemma \ref{lemma47} and Lemma \ref{lemma48}.
\end{proof}

\section{Asymptotics of the error function}

In this section we establish an asymptotic bound for the norm of the limiting
error function, i.e. for $\left\Vert H\left( \alpha,\cdot\right) \right\Vert
_{L_{\infty\left[ 0,\infty\right) }}$. This section is the most technical
part in this paper. Here, we use the generalized Watson Lemma (Laplace method
for integrals with large parameter) for deriving an asymptotic expansion used
to be later in the context. As it turns out, we need an higher order
asymptotics up to order $5$ involving the computation of certain rather
complicated defined constants. However, the main idea for deriving a lower
estimate is quite easy to see. Let us start, once again, with a diagram
(Figure \ref{Figure3}) involving the functions $\left\vert H\left(
\alpha,\cdot\right) \right\vert $ and $H_{1}\left( \alpha,\cdot\right) .$

\begin{figure}[th]
\begin{center}
\includegraphics[width=0.45\textwidth]{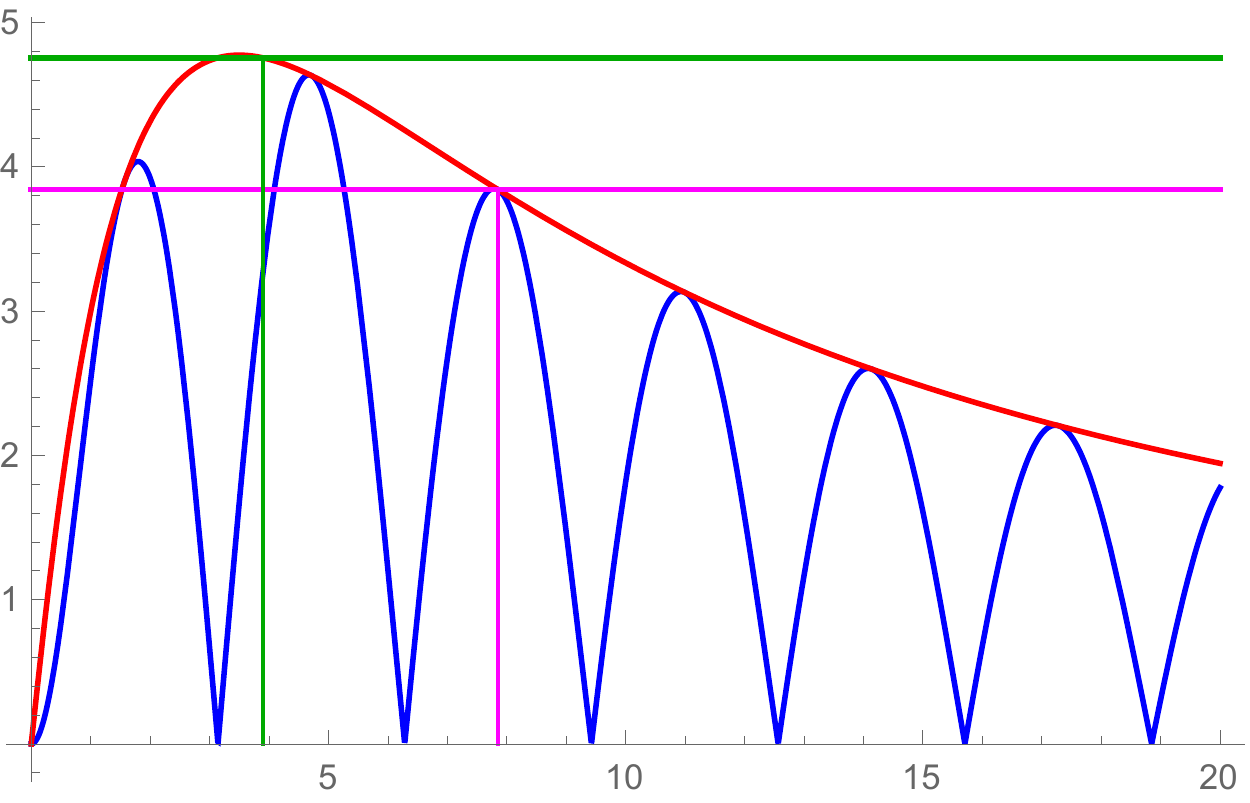} \hfill
\includegraphics[width=0.45\textwidth]{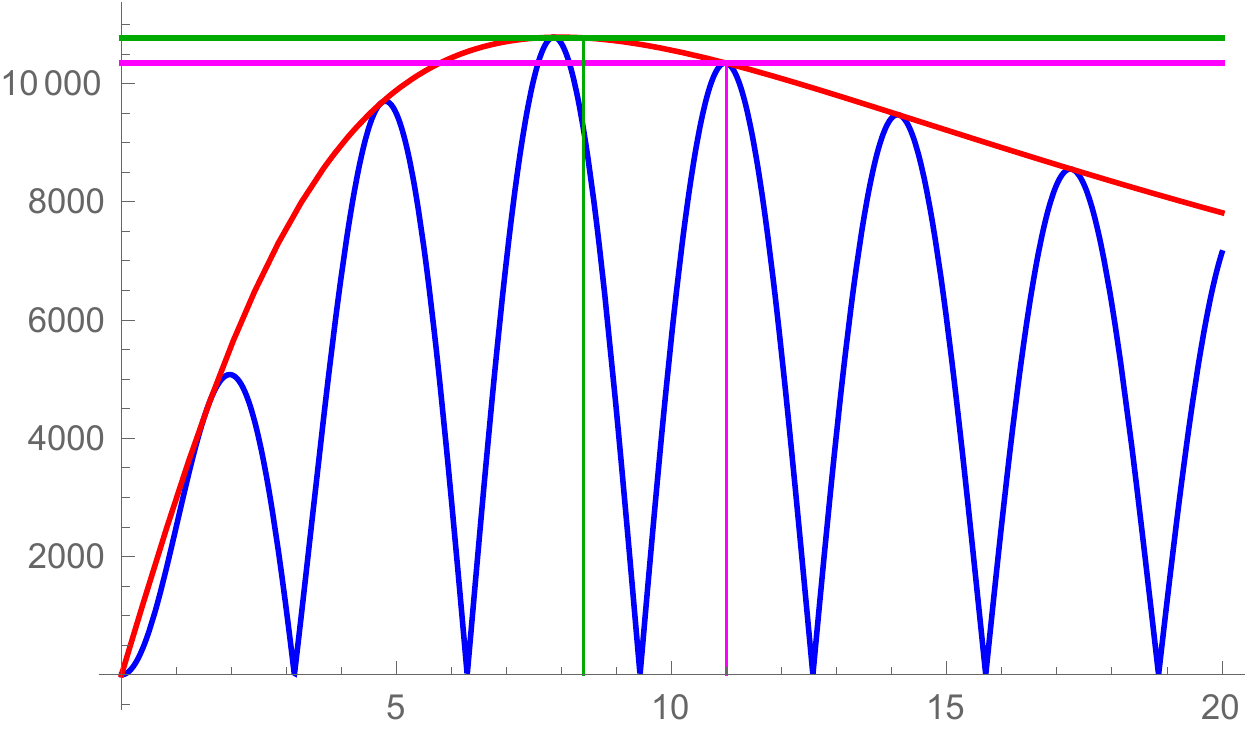}
\end{center}
\caption{The error functions $\left\vert H\left( \alpha,\cdot\right)
\right\vert $ and its envelope $H_{1}\left( \alpha,\cdot\right) $.}%
\label{Figure3}%
\end{figure}

\noindent
Figure \ref{Figure3} shows the functions $\left\vert H\left(
\alpha,\cdot\right) \right\vert $ and its envelope $H_{1}\left( \alpha
,\cdot\right) $ together with the point evaluations $H_{1}\left(
\alpha,\alpha\right) $ and $\left\vert H\left( \alpha,\beta\right)
\right\vert =H_{1}\left( \alpha,\beta\right) $, where $\beta=\beta\left(
\alpha\right) =\pi\left[ \frac{\alpha}{\pi}\right] +\frac{3}{2}\pi$ and
$\alpha=3.9$ and $\alpha=8.4$. Geometrically, the point $\beta$ is the
position of the first or the second relative maximum of $\left\vert H\left(
\alpha,\cdot\right) \right\vert $ on the right-hand side of $\alpha$, where
$H_{1}\left( \alpha,\cdot\right) $ appears to be descending. For growing
values of $\alpha$, the size of these maxima appear to be of the same
magnitude compared to the size $H_{1}\left( \alpha,\alpha\right) $. We use
both observations for the asymptotic analysis. First, we show that
$H_{1}\left( \alpha,\cdot\right) $ is descending at least for values
$x\geq\alpha$. Then, we derive the asymptotics for the local maximum in
$\left\vert H\left( \alpha,\beta\right) \right\vert $. It turns out that the
following integral inequality plays an essential role.

\begin{theorem}
\label{Theorem51}There exists a fixed constant $\alpha_{0}>0$ such that for
$\alpha\geq\alpha_{0}$,
\begin{equation}
R\left( \alpha,\alpha\right) =\int_{0}^{\infty}\frac{t^{\alpha+1}}%
{\sinh\alpha t}\frac{1}{1+t^{2}}dt-\int_{0}^{\infty}\frac{t^{\alpha}}%
{\sinh\alpha t}\frac{1}{1+t^{2}}dt>0. \label{mike01}%
\end{equation}
\end{theorem}

\begin{figure}[th]
\begin{center}
\includegraphics[width=0.6\textwidth]{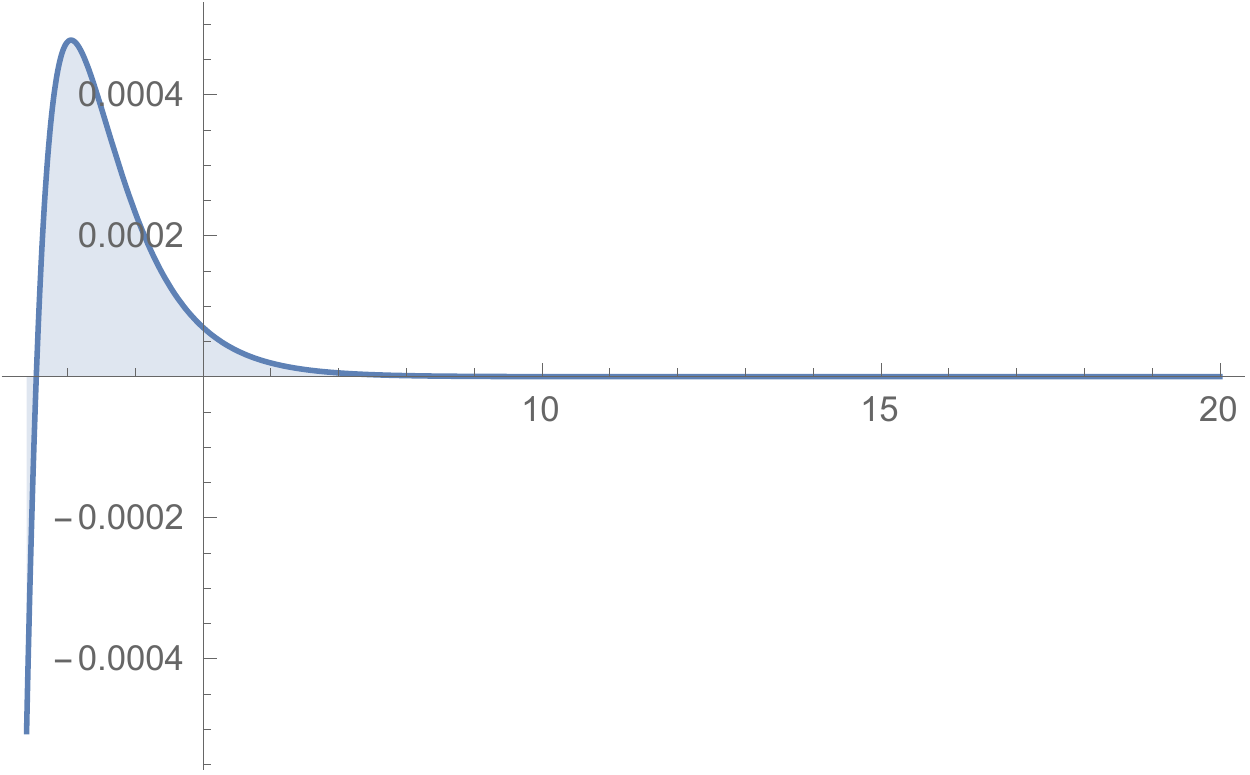}
\end{center}
\caption{Function $R\left( \cdot,\cdot\right) $ for $\alpha\in\left[ 2.4,
20 \right] $}%
\label{Figure4}%
\end{figure}

\noindent
We remark that (\ref{mike01}) is not true for all $\alpha_{0}>0.$
This can be seen out from Figure \ref{Figure4}. Also, for growing values of
$\alpha$, the positive magnitude becomes rather small. Numerical experiments
suggest that the minimal value for $\alpha_{0}$ such that (\ref{mike01})
becomes true, is somewhere in the interval $\left( 2.54288,2.54289\right) $.
However, since we are interested in an asymptotic expansion, the determination
of the exact size of the minimal value $\alpha_{0}$ is not important. From
Theorem \ref{Theorem51} we may derive our first desired property.

\begin{theorem}
\label{Theorem52}There exists a fixed constant $\alpha_{0}>0$ such that
$H_{1}\left( \alpha,\cdot\right) $ is decreasing, whenever $x\geq\alpha
\geq\alpha_{0}.$
\end{theorem}

\noindent
From Theorem \ref{Theorem52} we obtain the final asymptotics.

\begin{theorem}
\label{Theorem53} We have%
\[
\left\Vert H\left( \alpha,\cdot\right) \right\Vert _{L_{\infty}\left[
0,\infty\right) }=\frac{C\left( \alpha\right) }{1+2\alpha}\left(
1+o\left( 1\right) \right) ,\quad\alpha\rightarrow\infty.
\]
\end{theorem}

\noindent
We first establish Theorem \ref{Theorem52} by assuming that Theorem
\ref{Theorem51} holds true. Then, we present the proof for Theorem
\ref{Theorem51} which is completely independent of the forthcoming Lemmas
related to Theorem \ref{Theorem52}. Finally, we present the proof for Theorem
\ref{Theorem53}. Without proof, we first present the following

\begin{lemma}
\label{lemma51}Let $\alpha>0$ be fixed and $x>0$. Then $S\left(
\alpha,x\right) $ has the representation%
\[
S\left( \alpha,x\right) =\int_{0}^{\infty}\frac{t^{\alpha}\left(
t-\alpha\right) }{2\sinh t}\frac{x^{2}+\alpha^{2}}{x^{2}+t^{2}}dt.
\]
\end{lemma}

\begin{lemma}
\label{lemma52}Let $\alpha>0$ be fixed and $x>0$. Then
\[
\frac{d}{dx}H_{1}\left( \alpha,x\right) \leq-\frac{2}{x^{2}+\alpha^{2}%
}S\left( \alpha,x\right) .
\]
\end{lemma}

\begin{proof}
Using (\ref{Properties1}a) and by differentiating under the integral, we get%
\begin{align*}
\frac{d}{dx}H_{1}\left( \alpha,x\right)   & =\frac{d}{dx}\left( x^{\alpha
}F\left( \alpha,x\right) \right) \\
& =\alpha x^{\alpha-1}\left( \int_{0}^{\infty}\frac{t^{\alpha}}{\sinh
xt}\frac{dt}{1+t^{2}}-\frac{x}{\alpha}\int_{0}^{\infty}\frac{t^{\alpha+1}%
}{\sinh xt}\frac{\cosh xt}{\sinh xt}\frac{dt}{1+t^{2}}\right) \\
& \leq\alpha x^{\alpha-1}\left( \int_{0}^{\infty}\frac{t^{\alpha}}{\sinh
xt}\frac{dt}{1+t^{2}}-\frac{x}{\alpha}\int_{0}^{\infty}\frac{t^{\alpha+1}%
}{\sinh xt}\frac{dt}{1+t^{2}}\right) \\
& =-\alpha x^{\alpha-1}\left( \frac{x}{\alpha}F\left( \alpha+1,x\right)
-F\left( \alpha,x\right) \right) =-\alpha x^{\alpha-1}R\left(
\alpha,x\right) \\
& =-\alpha x^{\alpha-1}\frac{2}{\alpha}x^{1-\alpha}\frac{1}{x^{2}+\alpha^{2}%
}S\left( \alpha,x\right) =-\frac{2}{x^{2}+\alpha^{2}}S\left( \alpha
,x\right) .
\end{align*}
\end{proof}

\begin{lemma}
\label{lemma53}Let $\alpha>0$ be fixed and $x>0.$ Then
\begin{equation}
S\left( \alpha,x\right) =\frac{\alpha}{2}x^{\alpha-1}\left( x^{2}%
+\alpha^{2}\right) R\left( \alpha,x\right) \label{Thaler}%
\end{equation}
is an increasing function in $x$.
\end{lemma}

\begin{proof}
Using Lemma (\ref{lemma51}) and by differentiating under the integral again,
we get
\begin{align*}
\frac{d}{dx}S\left( \alpha,x\right)   & =\int_{0}^{\infty}\frac{t^{\alpha
}\left( t-\alpha\right) }{2\sinh t}\frac{\partial}{\partial x}\left(
\frac{x^{2}+\alpha^{2}}{x^{2}+t^{2}}\right) dt\\
& =\int_{0}^{\infty}\frac{xt^{\alpha}\left( t-\alpha\right) ^{2}}{\sinh
t}\frac{t+\alpha}{\left( x^{2}+t^{2}\right) ^{2}}dt > 0.
\end{align*}
\end{proof}

\noindent
The rescaling of $R\left( \alpha,\cdot\right) $ in Lemma
\ref{lemma53} is now extremely useful in proving Theorem \ref{Theorem52}.
Considering formula (\ref{Thaler}) contributes to my colleague, Dr. Maximilian
Thaler, for which I thank him.

\begin{proof}[Proof of Theorem \ref{Theorem52}.]
By assuming the validity of Theorem \ref{Theorem51} there exists some $\alpha
_{0}>0$, such that $R\left( \alpha,\alpha\right) >0$, $\forall \alpha\geq
\alpha_{0}$. From this fact and (\ref{Thaler}) we deduce $S\left( \alpha
,\alpha\right) =\alpha^{\alpha+2}R\left( \alpha,\alpha\right) >0$, $\forall 
\alpha\geq\alpha_{0}$. Now, combining Lemma \ref{lemma52} together with
Lemma \ref{lemma53}, we establish for all $x\geq\alpha\geq\alpha_{0}$,
\[
\frac{d}{dx}H_{1}\left( \alpha,x\right) \leq-\frac{2}{x^{2}+\alpha^{2}%
}S\left( \alpha,x\right) \leq-\frac{2}{x^{2}+\alpha^{2}}S\left(
\alpha,\alpha\right) <0.
\]
\end{proof}

\noindent
We turn now to the proof for Theorem \ref{Theorem51}. As before, we
derive several Lemmas.

\begin{lemma}
\label{lemma54}Let $\alpha>2$ be fixed and $x>0$. Then
\[
F_{1}\left( \alpha,x\right) \leq F\left( \alpha,x\right) \leq F_{2}\left(
\alpha,x\right) .
\]
\end{lemma}

\begin{proof}By some routine calculations we obtain the representation
\begin{equation}
F\left( \alpha,x\right) =2\sum_{n=0}^{\infty}\frac{1}{\left( 1+2n\right)
^{\alpha-1}}\int_{0}^{\infty}\frac{t^{\alpha}e^{-xt}}{\left( 1+2n\right)
^{2}+t^{2}}dt.\label{F1}%
\end{equation}
Since $Z\left( \alpha\right)$ is the well known zeta function, from
(\cite{Gradshteyn}, 9.522.2) we derive for $\alpha>1$,
\begin{equation}
Z\left( \alpha\right) \left( 2-2^{1-\alpha}\right) =2\sum_{n=0}^{\infty
}\frac{1}{\left( 1+2n\right) ^{\alpha}}.\label{F2}%
\end{equation}
Combining (\ref{F1}) together with (\ref{F2}), we obtain for $\alpha>2$ the
right-hand side in Lemma \ref{lemma54} by%
\begin{align*}
F\left( \alpha,x\right)   & \leq2\sum_{n=0}^{\infty}\frac{1}{\left(
1+2n\right) ^{\alpha-1}}\int_{0}^{\infty}\frac{t^{\alpha}e^{-xt}}{1+t^{2}%
}dt\\
& =\left( 2-\frac{1}{2^{\alpha-2}}\right) Z\left( \alpha-1\right) G\left(
\alpha,x\right) \\
& =F_{2}\left( \alpha,x\right) .
\end{align*}
Similarly, for $\alpha>0$, the left-hand side in Lemma \ref{lemma54} can be derived by%
\begin{align*}
F\left( \alpha,x\right)   & \geq2\sum_{n=0}^{\infty}\frac{1}{\left(
1+2n\right) ^{\alpha-1}}\int_{0}^{\infty}\frac{t^{\alpha}e^{-xt}}{\left(
1+2n\right) ^{2}+\left( 1+2n\right) ^{2}t^{2}}dt\\
& =\left( 2-\frac{1}{2^{\alpha}}\right) Z\left( \alpha+1\right) G\left(
\alpha,x\right) \\
& =F_{1}\left( \alpha,x\right) .
\end{align*}
\end{proof}

\begin{lemma}
\label{lemma55}Let $\alpha>2$. Then%
\begin{align*}
R\left( \alpha,\alpha\right)  & \geq\left( 2-\frac{1}{2^{\alpha+1}%
}\right) G\left( \alpha+1,\alpha\right) \\
& -\left( 2-\frac{1}{2^{\alpha-2}}\right) \left( 1+\frac{1}{2^{\alpha-1}%
}+\frac{1}{\alpha-2}\frac{1}{2^{\alpha-2}}\right) G\left( \alpha
,\alpha\right) .
\end{align*}
\end{lemma}

\begin{proof}
By using a routine estimate for the zeta function, namely
\[
1<Z\left( \alpha\right) <1+\frac{1}{2^{\alpha}}+\frac{1}{\alpha-1}\frac
{1}{2^{\alpha-1}},\quad\alpha>1,
\]
we combine this together with Lemma \ref{lemma54}. For $\alpha>2$ it then
follows%
\begin{align*}
R\left( \alpha,\alpha\right)  & =F\left( \alpha+1,\alpha\right) -F\left(
\alpha,\alpha\right) \\
& \geq F_{1}\left( \alpha+1,a\right) -F_{2}\left( \alpha,\alpha\right) \\
& =\left( 2-\frac{1}{2^{\alpha+1}}\right) Z\left( \alpha+2\right)
G\left( \alpha+1,\alpha\right) \\
& -\left( 2-\frac{1}{2^{\alpha-2}}\right) Z\left( \alpha-1\right)
G\left( \alpha,\alpha\right) \\
& \geq\left( 2-\frac{1}{2^{\alpha+1}}\right) \cdot1\cdot G\left(
\alpha+1,\alpha\right) \\
& -\left( 2-\frac{1}{2^{\alpha-2}}\right) \left( 1+\frac{1}{2^{\alpha-1}%
}+\frac{1}{\alpha-2}\frac{1}{2^{\alpha-2}}\right) G\left( \alpha
,\alpha\right) .
\end{align*}
\end{proof}

\begin{lemma}
\label{lemma56}Let $\alpha>0$ and $c\geq0$. Then, as $\alpha\rightarrow
\infty,$ we have the following asymptotics.
\begin{align*}
G\left( \alpha,\alpha\right)  & =\sqrt{\frac{2\pi}{\alpha}}e^{-\alpha
}\left( \frac{1}{2}-\frac{5}{24}\frac{1}{\alpha}+\frac{61}{576}\frac
{1}{\alpha^{2}}+O\left( \alpha^{-3}\right) \right) ,\\
G\left( \alpha+1,\alpha\right)  & =\sqrt{\frac{2\pi}{\alpha}}e^{-\alpha
}\left( \frac{1}{2}-\frac{5}{24}\frac{1}{\alpha}+\frac{205}{576}\frac
{1}{\alpha^{2}}+O\left( \alpha^{-3}\right) \right) ,\\
G\left( \alpha,\alpha+c\right)  & =\sqrt{\frac{2\pi}{\alpha}}e^{-\alpha
}\left( \frac{e^{-c}}{2}+O\left( \alpha^{-1}\right) \right) .
\end{align*}
\end{lemma}

\begin{proof}
We prove the relations with the generalized Watson Lemma. Let $\alpha>0$,
$k=0,1$ and $c\geq0.$ Then%
\begin{align*}
G\left( \alpha+k,\alpha+c\right)  & =\int_{0}^{\infty}\frac{t^{k}}%
{e^{ct}\left( 1+t^{2}\right) }e^{-\alpha\left( t-\log t\right) }dt\\
& =\int_{0}^{\infty}f_{k,c}\left( t\right) e^{-\alpha g\left( t\right)
}dt,
\end{align*}
with $f_{k,c}\left( t\right) =t^{k}/\left( e^{ct}\left( 1+t^{2}\right)
\right) $ and $g\left( t\right) =t-\log t.$ Before applying the Watson
Lemma, we have to split the integral in two parts $G\left( \alpha
+k,\alpha+c\right) =\int_{0}^{\infty}=\int_{1}^{\infty}+\int_{0}^{1},$
because $g$ has exactly one single minimum at $a=1$. After verifying the
conditions for the Watson Lemma (\cite{Olver}, Theorem 8.1) it allows us to
expand the integral $\int_{1}^{\infty}$ into an asymptotic series of the form%
\[
\int_{a}^{\infty}f_{k,c}\left( t\right) e^{-\alpha g\left( t\right)
}dt\simeq e^{-\alpha g\left( a\right) }\sum_{n=0}^{\infty}\Gamma\left(
\frac{n+\lambda}{\mu}\right) \frac{a_{n}^{\left( k,c\right) }}%
{\alpha^{\left( n+\lambda\right) /\mu}},\quad\alpha\rightarrow\infty,
\]
with certain coefficients $\lambda,\mu$ and $a_{n}^{\left( k,c\right) }$.
For the second integral $\int_{0}^{1}$ we have to apply a suitable
transformation before expanding it. It is worth mentioning, that in the
classical textbooks on asymptotic analysis (compare \cite{Olver}, p. 86)
there is no general formula for the coefficients $a_{n}$ available. Only the
first one or two coefficients are derived and as it can be easily checked,
they are of rather complicated nature. Surprisingly, in the newer literature
(\cite{OlverLozier}, Formula 2.3.18) one can find a remarkable easy
representation for these coefficients in terms of some residues as well as a
reference for its derivation, namely (in our context)%
\begin{equation}
a_{n}^{\left( k,c\right) }=\frac{1}{\mu}\left. \text{Res}\right\vert
_{t=a}\left( \frac{f_{k,c}\left( t\right) }{\left( g\left( t\right)
-g\left( a\right) \right) ^{\left( n+\lambda\right) /\mu}}\right) ,\quad
n=0,1,2,\ldots\label{residue}%
\end{equation}
We used a symbolic computation software for the computation of the residues in
(\ref{residue}), but we do not present the general outcome of these formulas.
This would fill several pages. However, since the calculations are of crucial
importance in the proof for Theorem \ref{Theorem51}, we present all relevant outputs. For
$k=0,1$ and $c=0$ we calculate%
\[%
\begin{array}
[c]{ll}%
a_{0}^{\left( k,0\right) }=\frac{1}{2\sqrt{2}}, & a_{3}^{\left( k,0\right)
}=\frac{45k^{3}-90k^{2}-90k+86}{270},\\
a_{1}^{\left( k,0\right) }=\frac{3k-1}{6}, & a_{4}^{\left( k,0\right)
}=\frac{36k^{4}-120k^{3}-96k^{2}+324k+61}{432\sqrt{2}},\\
a_{2}^{\left( k,0\right) }=\frac{6k^{2}-6k-5}{12\sqrt{2}}, & a_{5}^{\left(
k,0\right) }=\frac{189k^{5}-945k^{4}-315k^{3}+4683k^{2}+168k-3730}{11340}.
\end{array}
\]
For $c\geq0$, we compute $a_{0}^{\left( 0,c\right) }=e^{-c}\frac{1}%
{2\sqrt{2}}$ and $a_{1}^{\left( 0,c\right) }=-e^{-c}\frac{1+3c}{6}.$ With
$\lambda=1$ and $\mu=2$ we obtain for $\alpha\rightarrow\infty$,
\begin{align*}
\int_{1}^{\infty}f_{0,0}\left( t\right) e^{-\alpha g\left( t\right) }dt &
=\sqrt{\frac{2\pi}{\alpha}}e^{-\alpha}\left( \frac{1}{4}-\frac{1}{6\sqrt
{2\pi}}\frac{1}{\sqrt{\alpha}}-\frac{5}{48}\frac{1}{\alpha}+\frac{43}%
{135\sqrt{2\pi}}\frac{1}{\alpha\sqrt{\alpha}}\right. \\
& \left. +\frac{61}{1152}\frac{1}{\alpha^{2}}-\frac{746}{1143\sqrt{2\pi}%
}\frac{1}{\alpha^{2}\sqrt{\alpha}}+O\left( \alpha^{-3}\right) \right) ,\\
\int_{1}^{\infty}f_{1,0}\left( t\right) e^{-\alpha g\left( t\right) }dt &
=\sqrt{\frac{2\pi}{\alpha}}e^{-\alpha}\left( \frac{1}{4}+\frac{1}{3\sqrt
{2\pi}}\frac{1}{\sqrt{\alpha}}-\frac{5}{48}\frac{1}{\alpha}-\frac{49}%
{270\sqrt{2\pi}}\frac{1}{\alpha\sqrt{\alpha}}\right. \\
& \left. +\frac{205}{1152}\frac{1}{\alpha^{2}}+\frac{5}{567\sqrt{2\pi}}%
\frac{1}{\alpha^{2}\sqrt{\alpha}}+O\left( \alpha^{-3}\right) \right) ,\\
\int_{1}^{\infty}f_{0,c}\left( t\right) e^{-\alpha g\left( t\right) }dt &
=\sqrt{\frac{2\pi}{\alpha}}e^{-\alpha}\left( \frac{e^{-c}}{4}-e^{-c}%
\frac{1+3c}{6\sqrt{2\pi}}\frac{1}{\sqrt{\alpha}}+O\left( \alpha^{-1}\right)
\right) .
\end{align*}
Proceeding in the same way for the second integral $\int_{0}^{1}$, we compute%
\[%
\begin{array}
[c]{ll}%
a_{0}^{\left( k,0\right) }=\frac{1}{2\sqrt{2}}, & a_{3}^{\left( k,0\right)
}=-\frac{45k^{3}-90k^{2}-90k+86}{270},\\
a_{1}^{\left( k,0\right) }=-\frac{3k-1}{6}, & a_{4}^{\left( k,0\right)
}=\frac{36k^{4}-120k^{3}-96k^{2}+324k+61}{432\sqrt{2}},\\
a_{2}^{\left( k,0\right) }=\frac{6k^{2}-6k-5}{12\sqrt{2}}, & a_{5}^{\left(
k,0\right) }=-\frac{189k^{5}-945k^{4}-315k^{3}+4683k^{2}+168k-3730}{11340}.
\end{array}
\]
For $c\geq0$, we compute $a_{0}^{\left( 0,c\right) }=\frac{e^{-c}}{2\sqrt
{2}}$ and $a_{1}^{\left( 0,c\right) }=e^{-c}\frac{1+3c}{6}$. Again with
$\lambda=1$ and $\mu=2$ we obtain for $\alpha\rightarrow\infty$,
\begin{align*}
\int_{0}^{1}f_{0,0}\left( t\right) e^{-\alpha g\left( t\right) }dt &
=\sqrt{\frac{2\pi}{\alpha}}e^{-\alpha}\left( \frac{1}{4}+\frac{1}{6\sqrt
{2\pi}}\frac{1}{\sqrt{\alpha}}-\frac{5}{48}\frac{1}{\alpha}-\frac{43}%
{135\sqrt{2\pi}}\frac{1}{\alpha\sqrt{\alpha}}\right. \\
& \left. +\frac{61}{1152}\frac{1}{\alpha^{2}}+\frac{746}{1143\sqrt{2\pi}%
}\frac{1}{\alpha^{2}\sqrt{\alpha}}+O\left( \alpha^{-3}\right) \right) ,\\
\int_{0}^{1}f_{1,0}\left( t\right) e^{-\alpha g\left( t\right) }dt &
=\sqrt{\frac{2\pi}{\alpha}}e^{-\alpha}\left( \frac{1}{4}-\frac{1}{3\sqrt
{2\pi}}\frac{1}{\sqrt{\alpha}}-\frac{5}{48}\frac{1}{\alpha}+\frac{49}%
{270\sqrt{2\pi}}\frac{1}{\alpha\sqrt{\alpha}}\right. \\
& \left. +\frac{205}{1152}\frac{1}{\alpha^{2}}-\frac{5}{567\sqrt{2\pi}}%
\frac{1}{\alpha^{2}\sqrt{\alpha}}+O\left( \alpha^{-3}\right) \right) ,\\
\int_{0}^{1}f_{0,c}\left( t\right) e^{-\alpha g\left( t\right) }dt &
=\sqrt{\frac{2\pi}{\alpha}}e^{-\alpha}\left( \frac{e^{-c}}{4}+e^{-c}%
\frac{1+3c}{6\sqrt{2\pi}}\frac{1}{\sqrt{\alpha}}+O\left( \alpha^{-1}\right)
\right) .
\end{align*}
Collecting the results we finally arrive at the expansions in Lemma \ref{lemma56}.
\end{proof}

\begin{lemma}
\label{lemma57}There exists some $\alpha_{1}>0$, such that%
\[
G\left( \alpha+1,\alpha\right) -\left( 1+\frac{1}{\alpha^{3}}\right)
G\left( \alpha,\alpha\right) >0,\quad\forall\alpha\geq\alpha_{1}.
\]
\end{lemma}

\begin{proof}
From Lemma \ref{lemma56}, we calculate
\[
G\left( \alpha+1,\alpha\right) -G\left( \alpha,\alpha\right) =\sqrt
{\frac{2\pi}{\alpha}}e^{-\alpha}\left( \frac{1}{4\alpha^{2}}+O\left(
\alpha^{-3}\right) \right) ,\quad\alpha\rightarrow\infty.
\]
Now, combining the last expression together with Lemma \ref{lemma56}, we
obtain for $\alpha\rightarrow\infty$ the asymptotics
\begin{align*}
& G\left( \alpha+1,\alpha\right) -G\left( \alpha,\alpha\right) -\frac
{1}{\alpha^{3}}G\left( \alpha,\alpha\right) \\
& =\sqrt{\frac{2\pi}{\alpha}}e^{-\alpha}\left( \frac{1}{4\alpha^{2}%
}+O\left( \alpha^{-3}\right) \right) -\frac{1}{\alpha^{3}}\sqrt{\frac{2\pi
}{\alpha}}e^{-\alpha}\left( \frac{1}{2}+O\left( \alpha^{-1}\right) \right)
\\
& =\sqrt{\frac{2\pi}{\alpha}}e^{-\alpha}\frac{1}{4\alpha^{2}}\left(
1+O\left( \alpha^{-1}\right) \right) .
\end{align*}
The assertion now follows.
\end{proof}

\begin{proof}[Proof of Theorem \ref{Theorem51}.]
Let $\alpha>\max\left( 2, \alpha_{1}\right)$. Combining Lemma \ref{lemma55} together with Lemma \ref{lemma57}, we deduce%
\[
R\left( \alpha,\alpha\right) \geq\left( -\frac{1}{2^{\alpha+1}}%
+\frac{2-\frac{1}{2^{\alpha+1}}}{\alpha^{3}}+\frac{1}{2^{2\alpha-3}}%
-\frac{1-\frac{1}{2^{\alpha-1}}}{\left( \alpha-2\right) 2^{\alpha-3}%
}\right) G\left( \alpha,\alpha\right) .
\]
Since $G\left(\alpha,\alpha\right) >0, \forall \alpha>0$, an easy
calculation reveals that the remaining term in the last expression becomes
positive, at least for all $\alpha\geq\alpha_{0}=\max\left(14, \alpha_{1} \right) $.
\end{proof}

\noindent
We turn now to the proof for Theorem \ref{Theorem53}, again by
establishing some Lemmas. Without proof, we first present the following

\begin{lemma}
\label{lemma58}Let $\alpha>0$ and $\beta=\beta\left( \alpha\right)
=\pi\left[ \frac{\alpha}{\pi}\right] +\frac{3}{2}\pi.$ Then%
\[%
\begin{array}
[c]{ll}%
\text{(a)} & \alpha+\dfrac{\pi}{2}<\beta\leq\alpha+\dfrac{3}{2}\pi,\\
\text{(b)} & \left\vert H\left( \alpha,\beta\right) \right\vert
=H_{1}\left( \alpha,\beta\right) .
\end{array}
\]
\end{lemma}

\begin{lemma}
\label{lemma59}Let $\alpha>0$ and $c\geq0.$ Then%
\[
\frac{G\left( \alpha,\alpha+c\right) }{G\left( \alpha,\alpha\right)
}=e^{-c}\left( 1+O\left( \alpha^{-1}\right) \right) ,\quad\alpha
\rightarrow\infty.
\]
\end{lemma}

\begin{proof}
From Lemma \ref{lemma56}, we simply derive%
\begin{align*}
\frac{G\left( \alpha,\alpha+c\right) }{G\left( \alpha,\alpha\right) } &
=\frac{\sqrt{\frac{2\pi}{\alpha}}e^{-\alpha}\left( \frac{e^{-c}}{2}+O\left(
\alpha^{-1}\right) \right) }{\sqrt{\frac{2\pi}{\alpha}}e^{-\alpha}\left(
\frac{1}{2}+O\left( \alpha^{-1}\right) \right) }\\
& =e^{-c}\left( 1+O\left( \alpha^{-1}\right) \right).
\end{align*}
\end{proof}

\begin{lemma}
\label{lemma510}Let $\alpha>2.$ Then%
\[
H_{1}\left( \alpha,\alpha+\frac{3}{2}\pi\right) =H_{1}\left( \alpha
,\alpha\right) \left( 1+o\left( 1\right) \right) ,\quad\alpha
\rightarrow\infty.
\]
\end{lemma}

\begin{proof}
Using (\ref{Properties1}a), we obtain
\begin{equation}
\frac{H_{1}\left( \alpha,\alpha+\frac{3}{2}\pi\right) }{H_{1}\left(
\alpha,\alpha\right) }=\left( 1+\frac{\frac{3}{2}\pi}{\alpha}\right)
^{\alpha}\frac{F\left( \alpha,\alpha+\frac{3}{2}\pi\right) }{F\left(
\alpha,\alpha\right) }.\label{Formel3}%
\end{equation}
Next, using Lemma \ref{lemma54} together with a standard estimate for the
zeta function, we establish%
\begin{align}
\frac{F\left( \alpha,\alpha+\frac{3}{2}\pi\right) }{F\left( \alpha
,\alpha\right) } & \leq\frac{F_{2}\left( \alpha,\alpha+\frac{3}{2}%
\pi\right) }{F_{1}\left( \alpha,\alpha\right) }\nonumber\\
& \leq\frac{2-\frac{4}{2^{\alpha}}}{2-\frac{1}{2^{\alpha}}}\left( 1+\frac
{1}{\alpha-2}\right) \frac{G\left( \alpha,\alpha+\frac{3}{2}\pi\right)
}{G\left( \alpha,\alpha\right) },\label{Formel4}%
\end{align}
and%
\begin{align}
\frac{F\left( \alpha,\alpha+\frac{3}{2}\pi\right) }{F\left( \alpha
,\alpha\right) } & \geq\frac{F_{1}\left( \alpha,\alpha+\frac{3}{2}%
\pi\right) }{F_{2}\left( \alpha,\alpha\right) }\nonumber\\
& \geq\frac{2-\frac{1}{2^{\alpha}}}{2-\frac{4}{2^{\alpha}}}\left( 1-\frac
{1}{\alpha-1}\right) \frac{G\left( \alpha,\alpha+\frac{3}{2}\pi\right)
}{G\left( \alpha,\alpha\right) }.\label{Formel5}%
\end{align}
Now, combining (\ref{Formel3}), (\ref{Formel4}), (\ref{Formel5}) together with Lemma \ref{lemma59}, 
we establish the result.
\end{proof}

\begin{proof}[Proof of Theorem \ref{Theorem53}.]
For $\alpha\geq2$, it follows from Theorem \ref{Theorem41} that
\begin{equation}
\left\Vert H\left( \alpha,\cdot\right) \right\Vert _{L_{\infty}\left[
0,\infty\right) }\leq\left\Vert H_{1}\left( \alpha,\cdot\right) \right\Vert
_{L_{\infty}\left[ 0,\infty\right) }\leq\frac{C\left( \alpha\right)
}{1+2\alpha}\left( 1+\frac{2}{\sqrt{\alpha}}\right) .\label{Formel6}%
\end{equation}
For the reverse side, let $\varepsilon>0$ be arbitrary small. From Lemma
\ref{lemma510} we can find some $\alpha_{2}>0$, such that for $\alpha\geq
\alpha_{2}$,
\[
H_{1}\left( \alpha,\alpha+\frac{3}{2}\pi\right) \geq H_{1}\left(
\alpha,\alpha\right) \left( 1-\varepsilon\right) .
\]
Using Lemma \ref{lemma58}, Theorem \ref{Theorem52} and Theorem \ref{Theorem41},
we further obtain for $\alpha\geq \max\left(2, \alpha_{0}, \alpha_{2} \right) $ the estimate
\begin{align}
\left\Vert H\left( \alpha,\cdot\right) \right\Vert _{L_{\infty}\left[
0,\infty\right) } & \geq\left\vert H\left( \alpha,\beta\right) \right\vert
=H_{1}\left( \alpha,\beta\right) \nonumber\\
& \geq H_{1}\left( \alpha,\alpha+\frac{3}{2}\pi\right) \nonumber\\
& \geq H_{1}\left( \alpha,\alpha\right) \left( 1-\varepsilon\right)
\nonumber\\
& \geq\frac{C\left( \alpha\right) }{1+2\alpha}\left( 1-\frac{1}%
{\sqrt{\alpha}}\right) \left( 1-\varepsilon\right) .\label{Fromel7}%
\end{align}
Finally, combining (\ref{Formel6}) together with (\ref{Fromel7}), establishes
the result and we are finished.
\end{proof}

\section{Approximation polynomials in $L_{\infty}$}

\noindent
This section is devoted to an explicit construction for near best
approximation polynomials to $\left\vert x\right\vert ^{\alpha},\alpha>0$ in
the $L_{\infty}$ norm. The construction involves the polynomials
$P_{n}^{\left( 1\right) }$ and $P_{n}^{\left( 2\right) }$ together with
the Chebyshev polynomials $T_{n}$. The construction method is based on
numerical results. The
resulting formulas could indicate a general possible approach and structure
for the Bernstein constants $\Delta_{\alpha,\infty}$.

\medskip
\noindent
Let $\alpha>0$ be not an even integer.

\medskip
\noindent
First, let us collect some details on the interpolating polynomials
$P_{2n}^{\left( 1\right) }$. Recall, that the interpolation points are given
by $x_{j}^{\left( 2n\right) }=\cos\left( \left( j-\frac{1}{2}\right)
\pi/2n\right) $ for $j=1,2,\ldots,2n$ and $x_{0}^{\left( 2n\right) }=0$.
From Ganzburg (\cite{Ganzburg1}, Formulas 2.1, 2.7 and 4.14) it follows
\begin{align*}
& \lim_{n\rightarrow\infty}\left( 2n\right) ^{\alpha}\left\Vert \left\vert
x\right\vert ^{\alpha}-P_{2n}^{\left( 1\right) }\right\Vert _{L_{\infty
}\left[ -1,1\right] }\\
& =\frac{2}{\pi}\left\vert \sin\frac{\pi\alpha}{2}\right\vert \left\Vert
\int_{0}^{\infty}\frac{t^{\alpha-1}}{\cosh\left( t\right) }\frac{x^{2}\cos
x}{x^{2}+t^{2}}dt\right\Vert _{_{L_{\infty}\left[ 0,\infty\right) }}\\
& =\left\Vert \left\vert x\right\vert ^{\alpha}-G_{\alpha}\right\Vert
_{L_{\infty}\left[ 0,\infty\right) }\\
& =\frac{2}{\pi}\left\vert \sin\frac{\pi\alpha}{2}\right\vert \int
_{0}^{\infty}\frac{t^{\alpha-1}}{\cosh\left( t\right) }dt,
\end{align*}
where
\begin{equation}
G_{\alpha}\left( x\right) =\left\vert x\right\vert ^{\alpha}-\frac{2}{\pi
}\sin\frac{\pi\alpha}{2}\int_{0}^{\infty}\frac{t^{\alpha-1}}{\cosh\left(
t\right) }\frac{x^{2}\cos x}{x^{2}+t^{2}}dt \label{GanzburgEntire1}%
\end{equation}
is an entire function of exponential type $1$ that interpolates $\left\vert
x\right\vert ^{\alpha}$ at the nodes $\left\{ \left( k+\frac{1}{2}\right)
\pi:k\in\mathbb{Z}\right\} \cup\left\{ 0\right\} $. There also exists
(\cite{Ganzburg1}, Formula 4.15) a representation for $G_{\alpha}$ as an
interpolating series, similar to formula (\ref{Entire5}) in Theorem
\ref{Theorem33}.

\medskip
\noindent
By an analogue method as that was used in the proof for
Theorem \ref{Theorem32} one can show that uniformly on compact subsets of
$\left[ 0,\infty\right) $ we have the scaled limit%
\begin{equation}
\lim_{n\rightarrow\infty}\left( 2n\right) ^{\alpha}P_{2n}^{\left( 1\right)
}\left( \frac{x}{2n}\right) =G_{\alpha}\left( x\right) .
\label{GanzburgEntire2}%
\end{equation}

\noindent
Now, based on numerical computations, we made the following
observations. For all $\alpha>0$ not an even integer we find that,
beginning with the second positive note, all interpolation points of the best
approximation polynomials $P_{2n}^{\ast}$ are located somewhere between two
consecutive interpolation points for the $P_{2n}^{\left( 1\right) }$ and
$P_{2n}^{\left( 2\right) }$ polynomials. See Figure \ref{figure5}.

\begin{figure}[th]
\begin{center}
\includegraphics[width=0.75\textwidth]{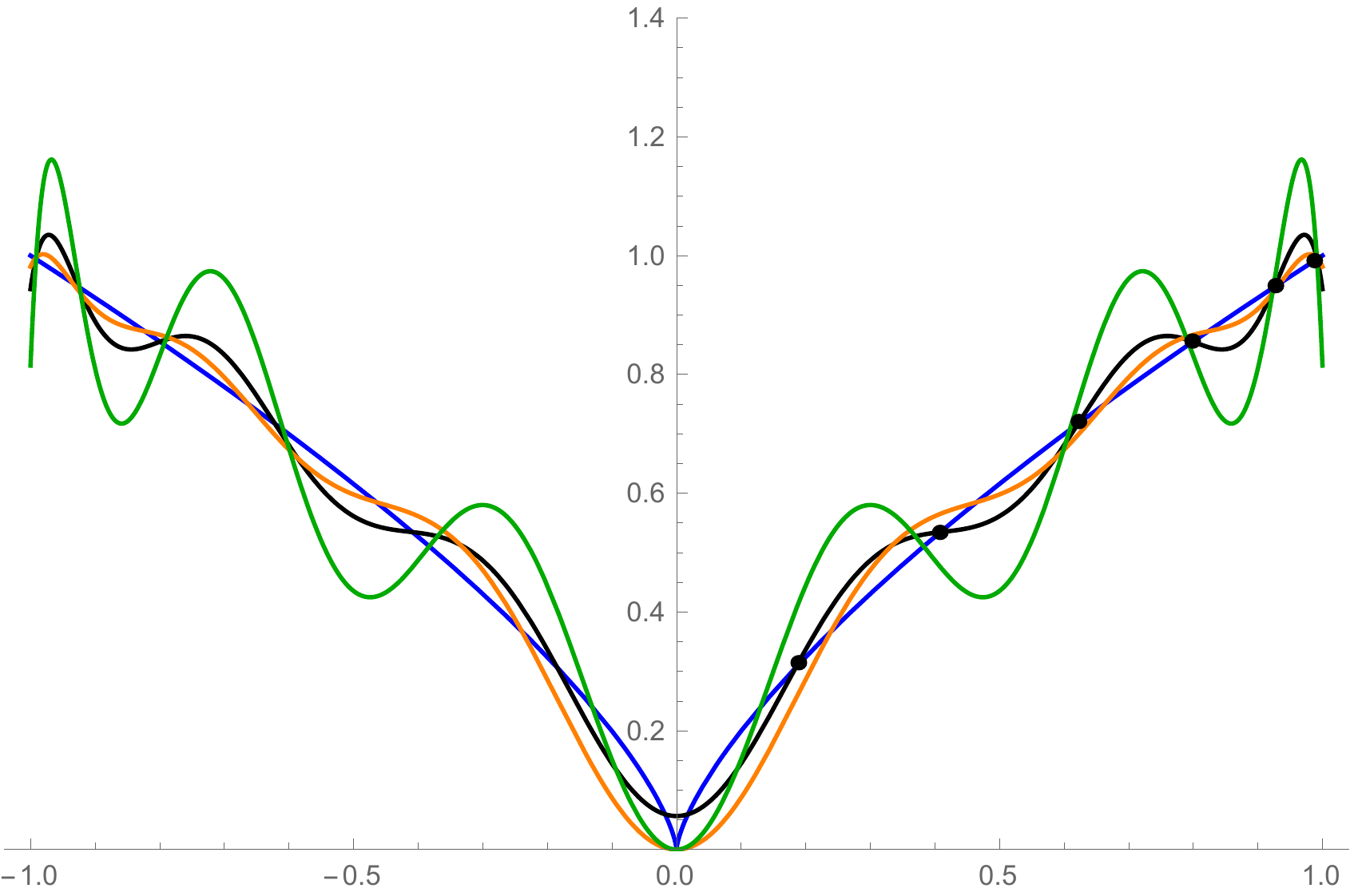}
\end{center}
\caption{Interpolation points for the best approximation to $\left\vert
x\right\vert ^{\alpha}$. }%
\label{figure5}%
\end{figure}

\medskip
\noindent
It is well known that $\left[ 1,x,\ldots,x^{n};x^{\alpha / 2} \right]$ is an
hypernormal Haar space of dimension $n+2$ on the interval $\left[ 0,1\right]
$, see (\cite{Varga2}, p. 199). Consequently it follows that we have always an
alternation point at $x=0$. Thus we cannot expect to perform in the quality of
best approximation solely by using the polynomials $P_{2n}^{\left( 1\right)
}$ and $P_{2n}^{\left( 2\right) }$, since both of them interpolate at $x=0$.
Thus we consider the following polynomials
\begin{align}
P_{2n}^{\left( 3 \right)} \left( x\right)  & =c_{1,\alpha}P_{2n}^{\left( 1\right) }\left(
x\right) +\left( 1-c_{1,\alpha}\right) P_{2n}^{\left( 2\right) }\left(
x\right) \nonumber\\
& +\frac{2}{\pi}\sin\frac{\pi\alpha}{2}c_{2,\alpha}\frac{\left( -1\right)
^{n}}{\left( 2n\right) ^{\alpha}}\frac{T_{2n+1}\left( x\right) }{\left(
2n+1\right) x}, \label{Nearbest}%
\end{align}
where $c_{1,\alpha}$ and $c_{2,\alpha}$ are numerical constants, depending
only on $\alpha$. As we see later, for good choices of $c_{1,\alpha}$ and
$c_{2,\alpha}$ the linear combination of $P_{2n}^{\left( 1\right) }$ and
$P_{2n}^{\left( 2\right) }$ results in a polynomial with almost all the same
interpolation points as its best approximation $P_{2n}^{\ast}$, while at the
same time the last term in (\ref{Nearbest}) establishes the alternation
property at $x=0$ and leaves the new interpolation points largely unchanged.

\medskip
\noindent
Since we are interested into the asymptotic behavior of the
polynomials $P_{2n}^{\left( 3 \right)}$ we directly pass to the resulting scaled limit. From
Theorem \ref{Theorem32}, formulas (\ref{Entire4}), (\ref{GanzburgEntire1}), (\ref{GanzburgEntire2}) and Lemma \ref{lemma36}, it follows that uniformly on compact subsets of
$\left[ 0,\infty\right) $ we have%
\begin{align}
\lim_{n\rightarrow\infty}\left( 2n\right) ^{\alpha}P_{2n}^{\left( 3 \right)} \left( \frac
{x}{2n}\right)  & =\left\vert x\right\vert ^{\alpha}-\frac{2}{\pi}\sin
\frac{\pi\alpha}{2}\left( c_{1,\alpha}\int_{0}^{\infty}\frac{t^{\alpha-1}%
}{\cosh t}\frac{x^{2}\cos x}{x^{2}+t^{2}}dt\right. \nonumber\\
& +\left. \left( 1-c_{1,\alpha}\right) \int_{0}^{\infty}\frac{t^{\alpha}%
}{\sinh t}\frac{x\sin x}{x^{2}+t^{2}}dt-c_{2,\alpha}\frac{\sin x}{x}\right) .
\label{BestEntire}%
\end{align}
Thus, we try to numerically minimize the quantity%
\begin{align*}
& \left\Vert c_{1,\alpha}\int_{0}^{\infty}\frac{t^{\alpha-1}}{\cosh t}%
\frac{x^{2}\cos x}{x^{2}+t^{2}}dt\right. \\
& \left. +\left( 1-c_{1,\alpha}\right) \int_{0}^{\infty}\frac{t^{\alpha}%
}{\sinh t}\frac{x\sin x}{x^{2}+t^{2}}dt-c_{2,\alpha}\frac{\sin x}%
{x}\right\Vert _{L_{\infty}\left[ 0,\infty\right) }.
\end{align*}
For the moment, we cannot present an explicit formula for the constants $c_{1,\alpha}$ and
$c_{2,\alpha}$, but based on numerical calculations, we present the following
numerical table.
\[%
\begin{tabular}
[c]{|c|c|c||c|c|c|}\hline
$\alpha$ & $c_{1,\alpha}$ & $c_{2,\alpha}$ & $\alpha$ & $c_{1,\alpha}$ &
$c_{2,\alpha}$\\\hline
0.1 & 0.43 & 4.40 & 1.1 & 0.25 & 0.44\\
0.2 & 0.39 & 2.05 & 1.2 & 0.22 & 0.42\\
0.3 & 0.36 & 1.32 & 1.3 & 0.22 & 0.41\\
0.4 & 0.34 & 0.97 & 1.4 & 0.21 & 0.41\\
0.5 & 0.33 & 0.78 & 1.5 & 0.19 & 0.41\\
0.6 & 0.31 & 0.65 & 1.6 & 0.17 & 0.42\\
0.7 & 0.30 & 0.57 & 1.7 & 0.15 & 0.44\\
0.8 & 0.28 & 0.51 & 1.8 & 0.12 & 0.46\\
0.9 & 0.27 & 0.48 & 1.9 & 0.10 & 0.49\\
1.0 & 0.26 & 0.45 & & & \\\hline
\end{tabular}
\
\]

\noindent
Using these numerical values, we present some illustrations for the
$P_{n}^{\left( 3 \right)}$ polynomials from (\ref{Nearbest}). In Figure \ref{Figure6} we present
the polynomials $P_{4}^{\left( 3 \right)}$, $P_{8}^{\left( 3 \right)}$ together with the best approximations
$P_{4}^{\ast}$, $P_{8}^{\ast}$ and $\alpha=0.5$. The same is done in Figure
\ref{Figure7} for $\alpha=1.0$.

\begin{figure}[th]
\begin{center}
\includegraphics[width=0.45\textwidth]{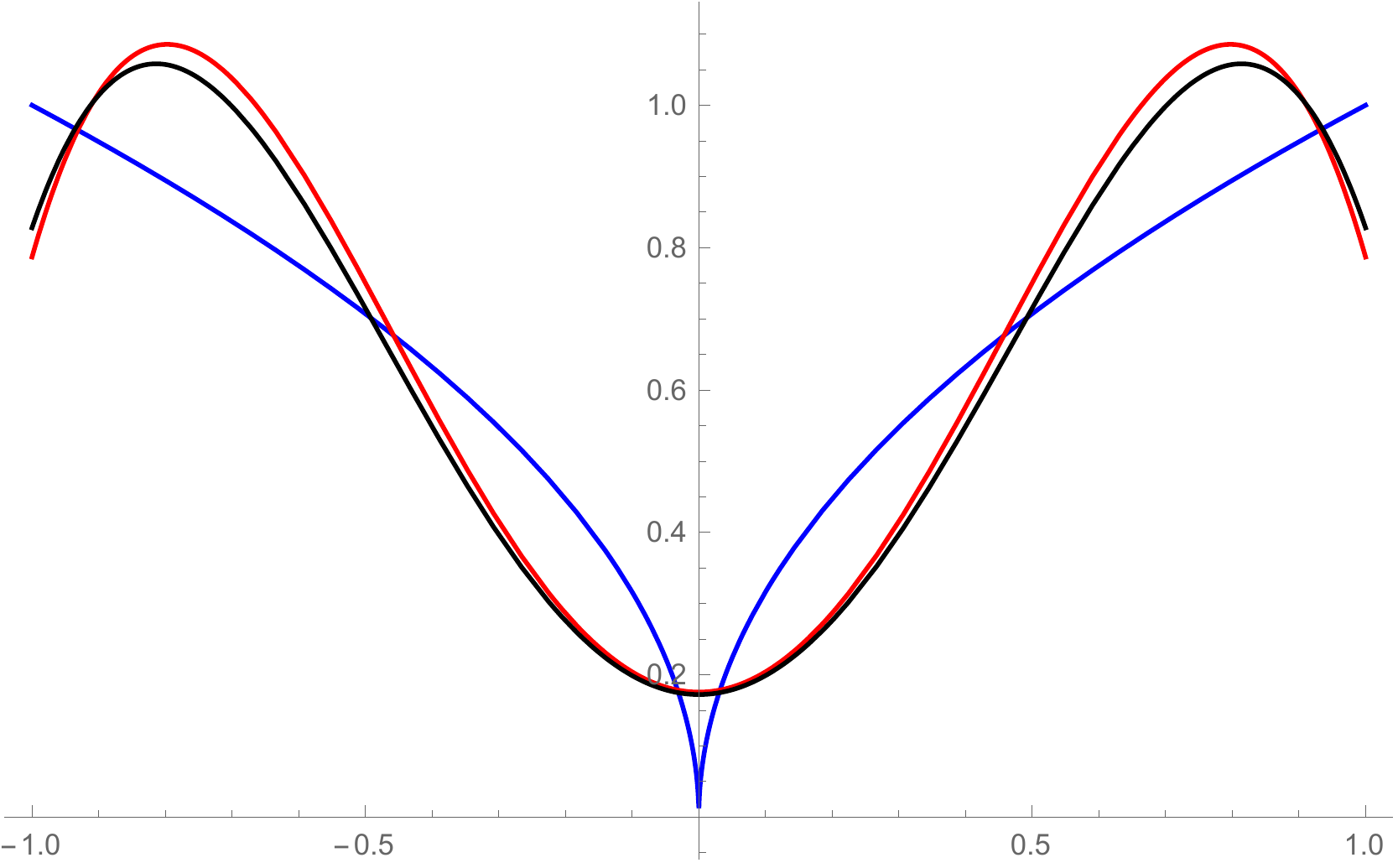} \hfill
\includegraphics[width=0.45\textwidth]{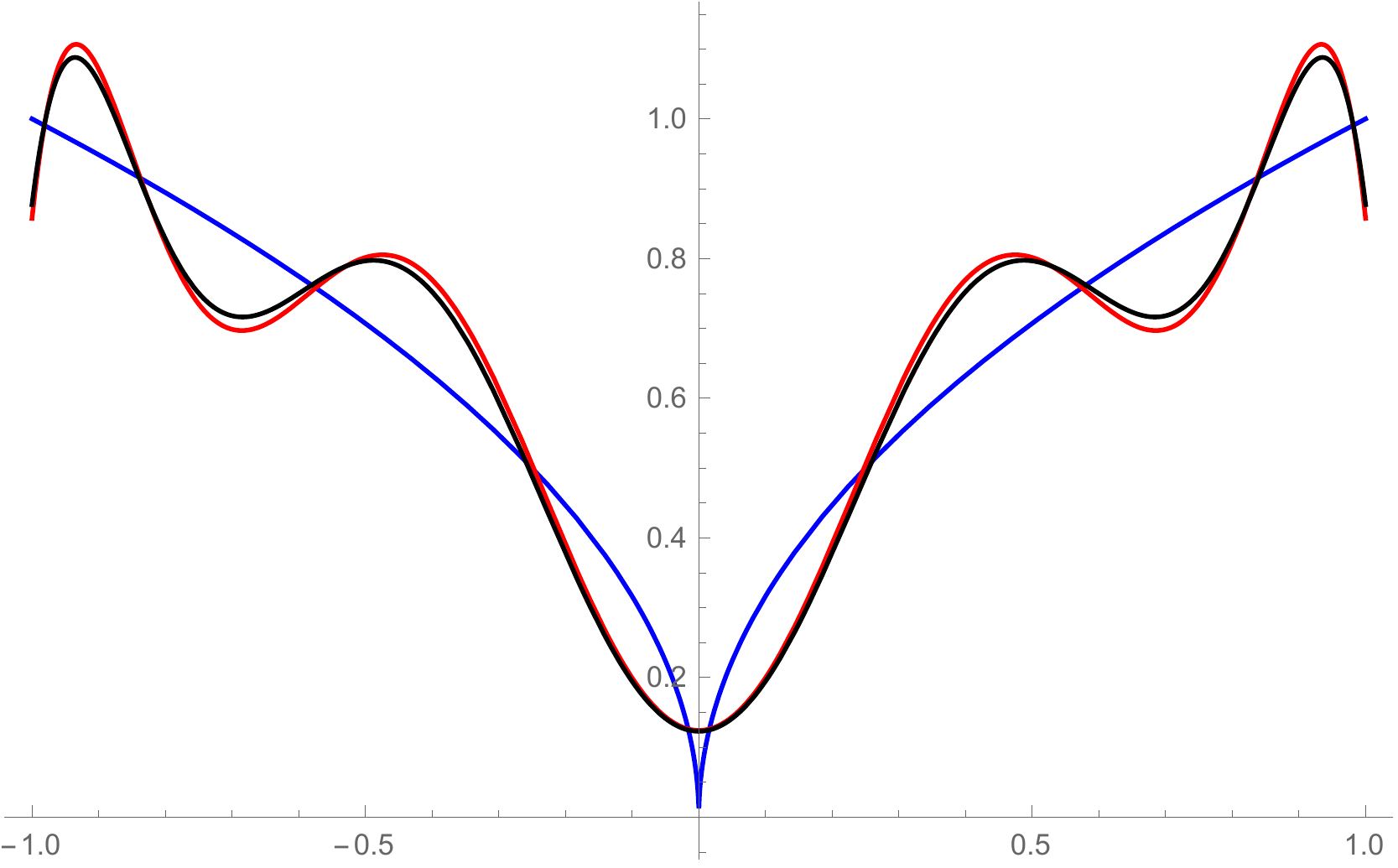}
\end{center}
\caption{$\alpha=0.5$. Polynomials $P_{4}^{\left( 3 \right)}$, $P_{4}^{\ast}$ and $P_{8}^{\left( 3 \right)}$,
$P_{8}^{\ast}$.}%
\label{Figure6}%
\end{figure}

\begin{figure}[th]
\begin{center}
\includegraphics[width=0.45\textwidth]{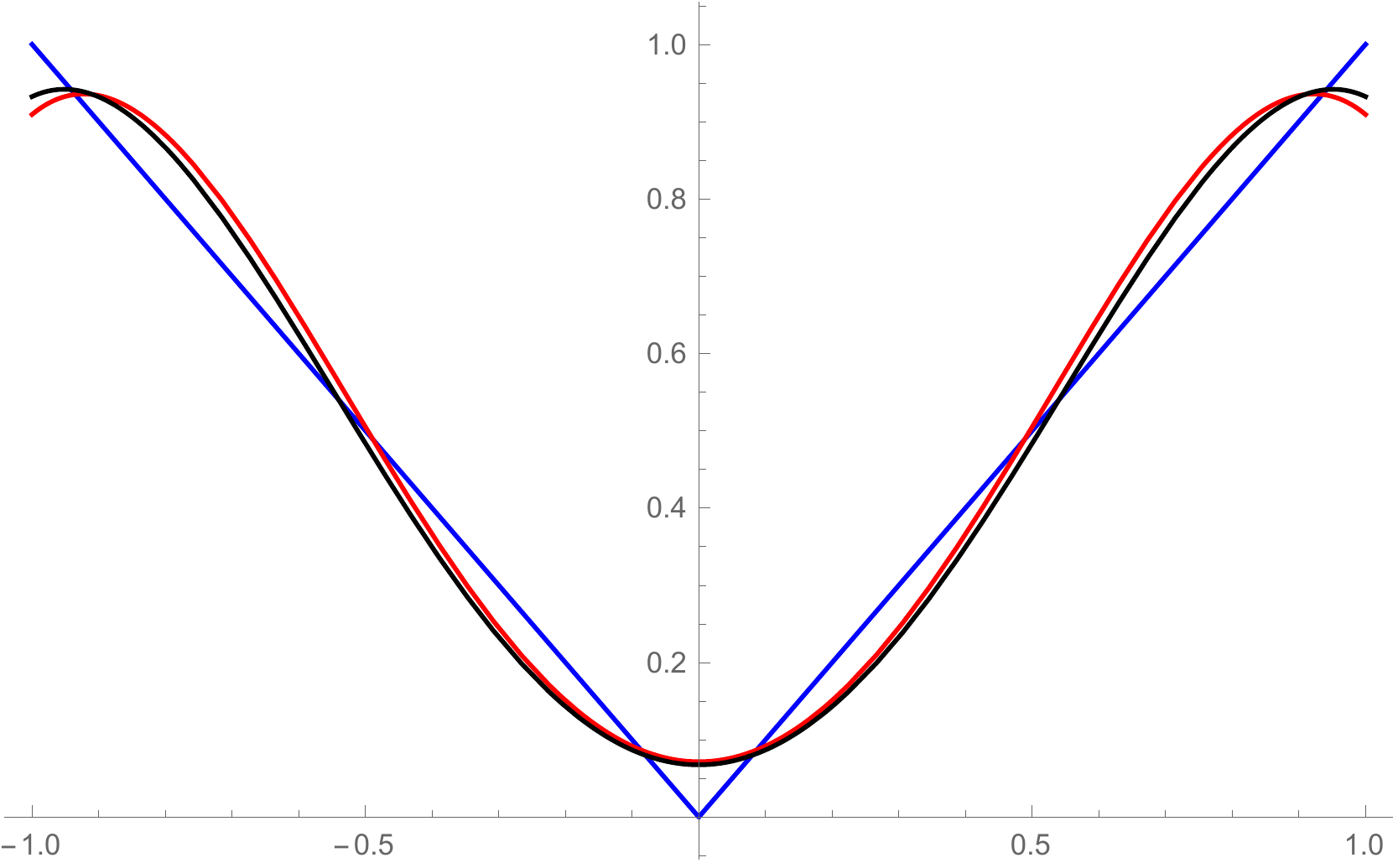} \hfill
\includegraphics[width=0.45\textwidth]{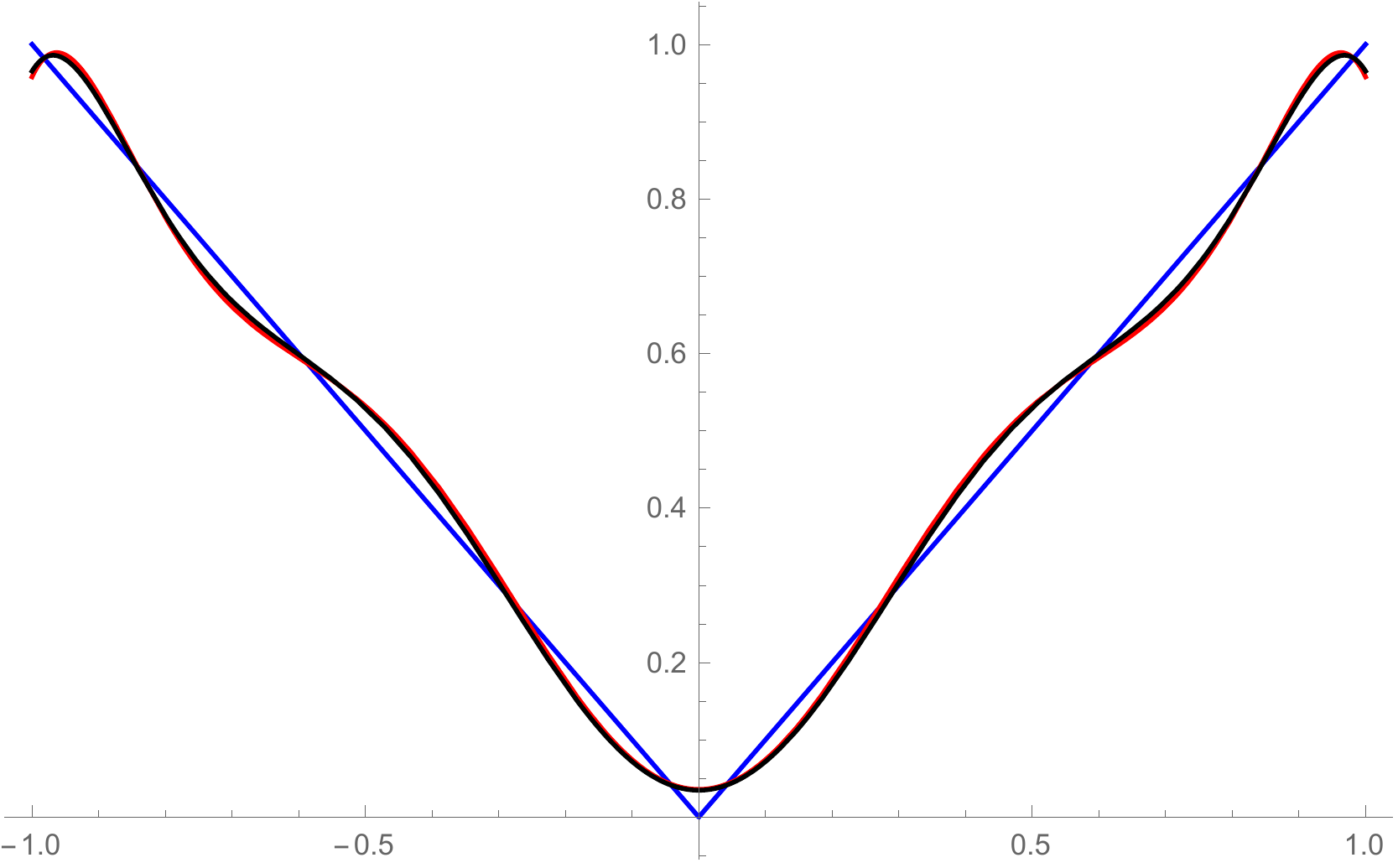}
\end{center}
\caption{$\alpha=1.0$. Polynomials $P_{4}^{\left( 3 \right)}$, $P_{4}^{\ast}$ and $P_{8}^{\left( 3 \right)}$,
$P_{8}^{\ast}$.}%
\label{Figure7}%
\end{figure}

\medskip\noindent We also tried to find some approximations for the minimizing
best entire functions $H_{\alpha}^{\ast}$ defined by%
\begin{align*}
& \Delta_{\infty,\alpha}=\left\Vert \left\vert x\right\vert ^{\alpha
}-H_{\alpha}^{\ast}\right\Vert _{L_{\infty}\left[ 0,\infty\right) }\\
& =\inf\left\{ \left\Vert \left\vert x\right\vert ^{\alpha}-H\right\Vert
_{L_{\infty}\left( \mathbb{R}\right) }:H\text{ is entire of exponential
type}\leq1\right\} .
\end{align*}
Especially we are interested into the locations of its corresponding
interpolation points. Recall, that from (\cite{Lubinsky1}, \cite{Lubinsky2}) it
follows, that uniformly on compact subsets of $\mathbb{C}$ we have%
\begin{equation}
\lim_{n\rightarrow\infty}\left( 2n\right) ^{\alpha}P_{2n}^{\ast}\left(
\frac{z}{2n}\right) =H_{\alpha}^{\ast}\left( z\right) , \label{numapprox}%
\end{equation}
There is also a representation for $H_{\alpha}^{\ast}$ as an interpolation
series with (unknown) interpolation points $0<x_{1}^{\ast}<x_{2}^{\ast}%
<x_{3}^{\ast}<\cdots$. However, it is known (\cite{Lubinsky2}, Theorem 1.1)
that%
\[
x_{j}^{\ast}\in\left[ \left( j-\frac{3}{2}\right) \pi,\left( j-\frac{1}%
{2}\right) \pi\right] ,\quad\forall j\geq2.
\]
Moreover, from (\cite{Lubinsky2}, Formulas 1.6 and 1.7) it follows that there
exists alternation points $0=y_{0}^{\ast}<y_{1}^{\ast}<y_{2}^{\ast}<\cdots$
with%
\[
\left\vert y_{j}^{\ast}\right\vert ^{\alpha}-H_{\alpha}^{\ast}\left( \pm
y_{j}^{\ast}\right) =\left( -1\right) ^{j+\overline{\alpha/2}}\left\Vert
\left\vert x\right\vert ^{\alpha}-H_{\alpha}^{\ast}\right\Vert _{L_{\infty
}\left( \mathbb{R}\right) },
\]
where $\overline{\alpha/2}$ is the least integer exceeding $\alpha/2$. For the
alternation points it is also known that%
\[
y_{j}^{\ast}\in\left[ \left( j-1\right) \pi,j\pi\right] ,\quad\forall
j\geq1.
\]
We use formula (\ref{BestEntire}) as an approximation for $H_{\alpha}^{\ast}$.
In Figure \ref{Figure8} we present some illustrations from (\ref{BestEntire})
for $\alpha=0.5$ and $\alpha=1.0$. In Figure \ref{Figure9} we illustrate the near
equioscillating behavior of the error term in (\ref{BestEntire}), again for
$\alpha=0.5$ and $\alpha=1.0$, and we compare the maximal error magnitude with the
corresponding numerical values for the Bernstein constants%
\begin{align*}
\Delta_{\infty, 0.5} & =0.348648\ldots,\\
\Delta_{\infty, 1} & =0.280169\ldots
\end{align*}
The values for the Bernstein constants are taken from (\cite{Varga2}, Table 1.1).

\begin{figure}[th]
\begin{center}
\includegraphics[width=0.45\textwidth]{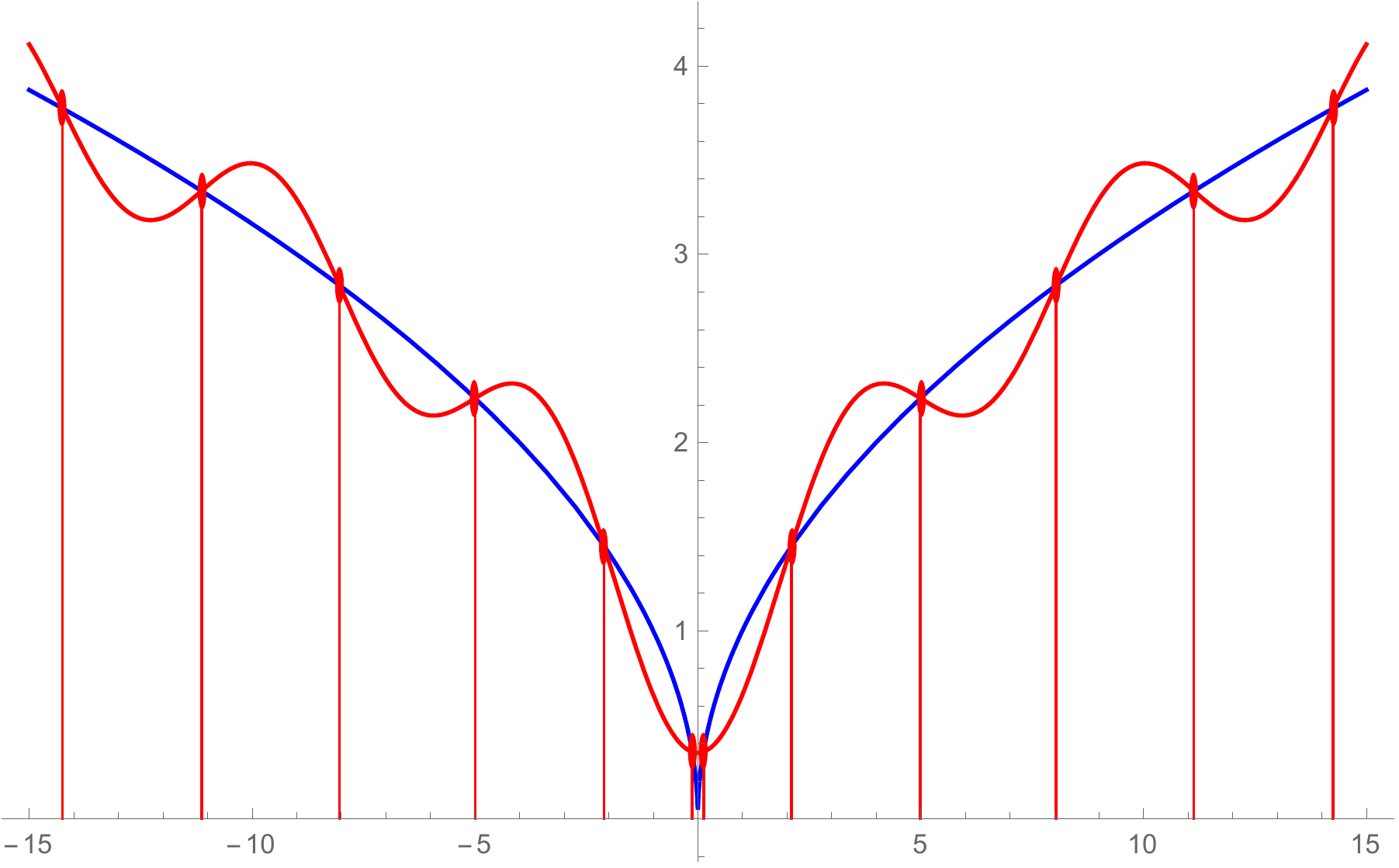} \hfill
\includegraphics[width=0.45\textwidth]{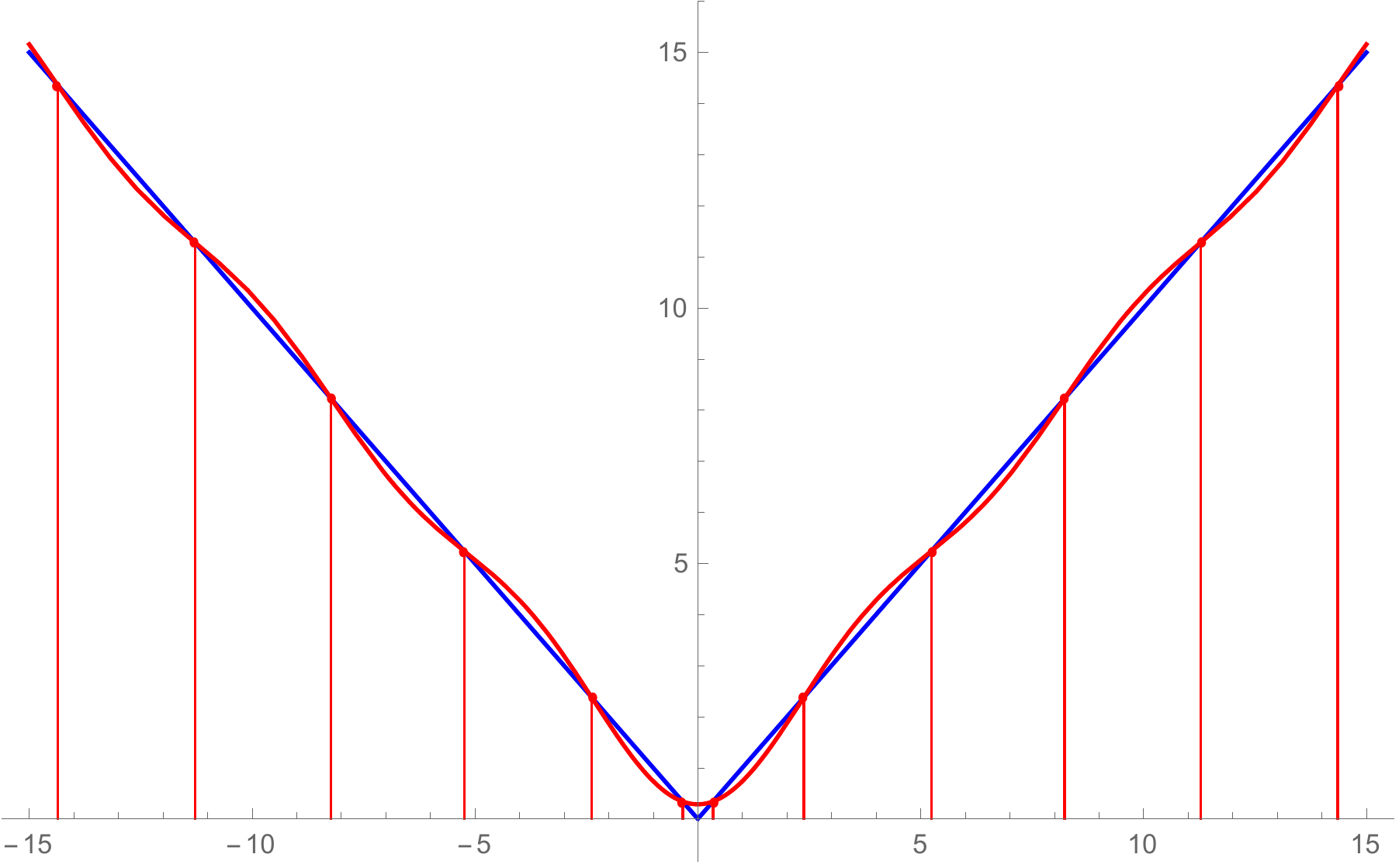}
\end{center}
\caption{Approximations for best entire functions $H_{\alpha}^{\ast}$ of
exponential type $1$.}%
\label{Figure8}%
\end{figure}

\begin{figure}[th]
\begin{center}
\includegraphics[width=0.45\textwidth]{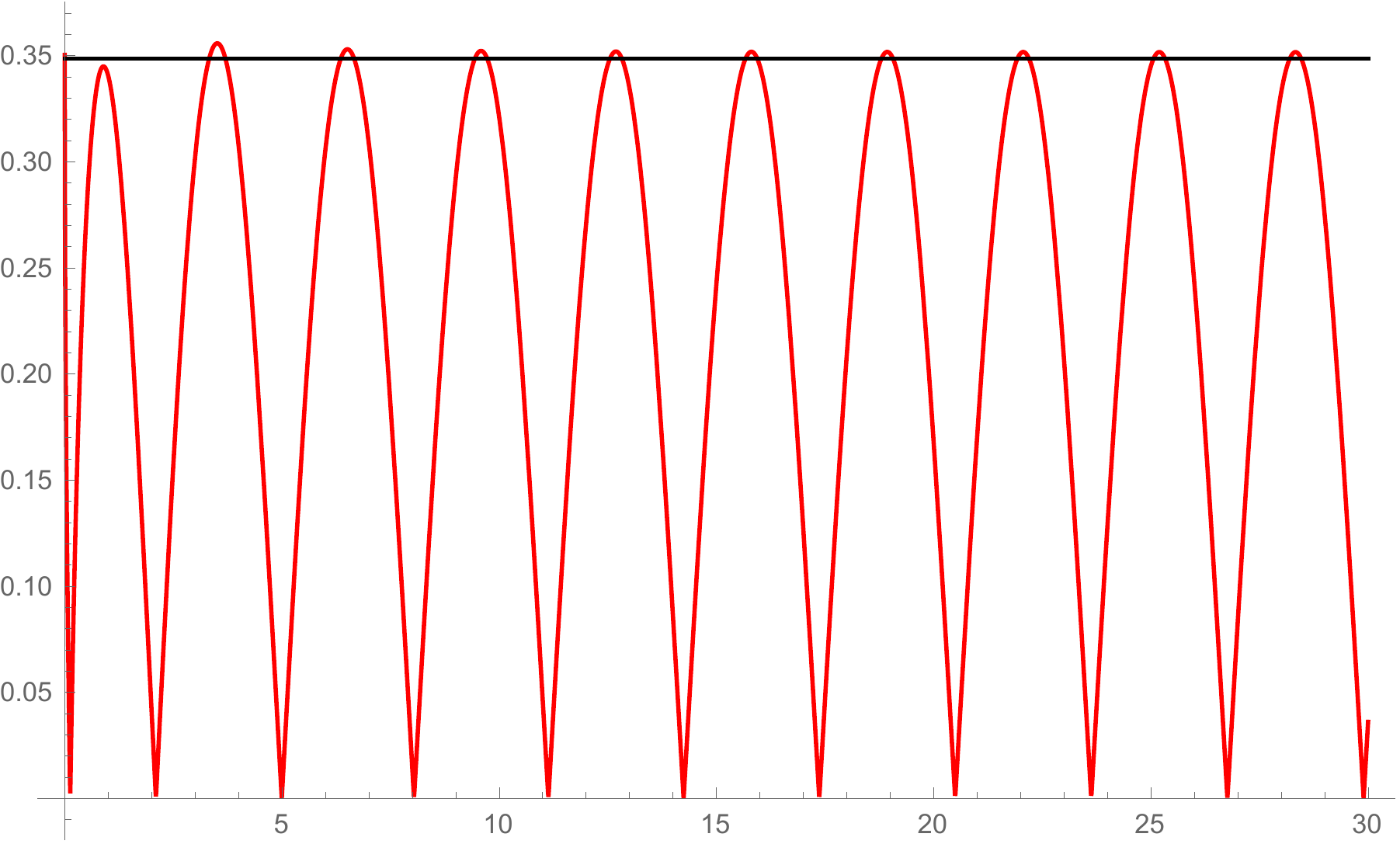} \hfill
\includegraphics[width=0.45\textwidth]{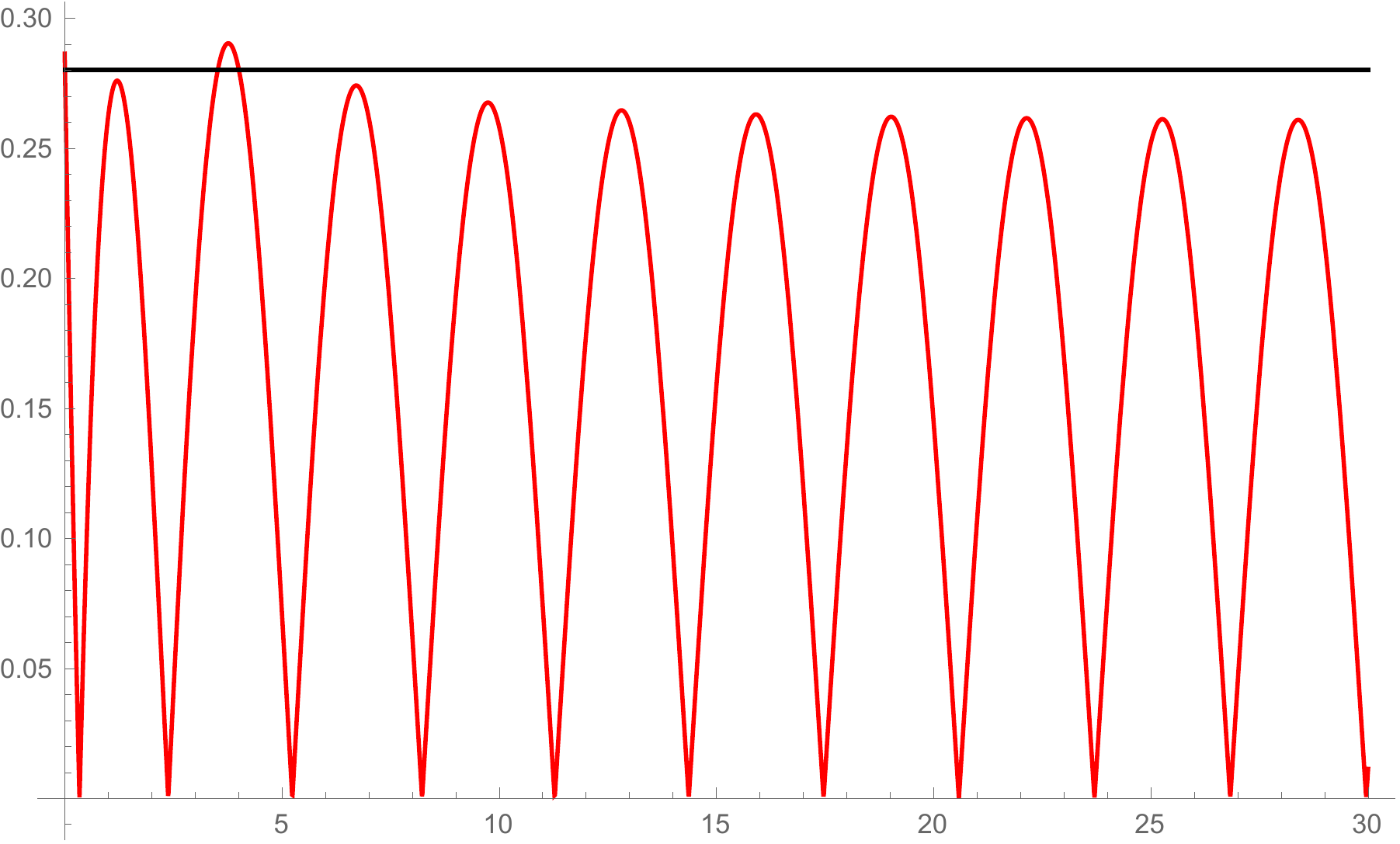}
\end{center}
\caption{Nearly equioscillation property of approximation for $H_{\alpha}^{\ast}$ together with
$\Delta_{\infty, 0.5}$ and $\Delta_{\infty, 1}$.}%
\label{Figure9}%
\end{figure}

\medskip\noindent In the following table we present the approximations for the
best interpolation points $x_{j}^{\ast}$ for $j=1,\ldots,10$ from (\ref{BestEntire}), respectively from Figure \ref{Figure8}.
\[%
\begin{tabular}
[c]{|c|c|c|c|c|c|c|c|c|c|c|}\hline
$\alpha$ & $x_{1}^{\ast}$ & $x_{2}^{\ast}$ & $x_{3}^{\ast}$ & $x_{4}^{\ast}$ &
$x_{5}^{\ast}$ & $x_{6}^{\ast}$ & $x_{7}^{\ast}$ & $x_{8}^{\ast}$ &
$x_{9}^{\ast}$ & $x_{10}^{\ast}$\\\hline
0.5 & 0.13 & 2.10 & 4.99 & 8.04 & 11.13 & 14.25 & 17.37 & 20.50 & 23.63 &
26.76\\
0.8 & 0.25 & 2.30 & 5.15 & 8.16 & 11.22 & 14.32 & 17.43 & 20.55 & 23.67 &
26.80\\
1.0 & 0.34 & 2.38 & 5.24 & 8.23 & 11.28 & 14.36 & 17.47 & 20.58 & 23.70 &
26.83\\\hline
\end{tabular}
\]

\noindent The last table suggests that, for small positive values $\alpha$, all
interpolation points are slightly shifted to the left. Apparently this effect
becomes greater for those interpolation points which are located closer to the
origin. On the other hand, the values suggest that
\[
x_{n+1}^{\ast}-x_{n}^{\ast}\rightarrow\pi,\quad n\rightarrow\infty,
\]
from below.

\medskip
\noindent
Finally, we remark that the overall quality of the
$P_{n}^{\left(  3\right)  }$ polynomials appears to be very encouraging in search for some representations of the Bernstein constants. Their approximation properties with respect to the corresponding best approximation polynomials $P_{n}^{\ast}$ are of high quality, even for small values of $n$. Thus, formula \ref{BestEntire} though it is at the present time not in its full explicit form, appears to be an important step towards a possible representation for the Bernstein constants $\Delta_{\infty,\alpha}$.

Michael Revers

Department of Mathematics

University Salzburg

Hellbrunnerstrasse 34

A-5020 Salzburg

AUSTRIA

E-Mail: michael.revers@sbg.ac.at
\end{document}